\documentclass[10pt, reqno]{amsart}

\usepackage{amscd,amsmath}
\usepackage{mathrsfs}
\usepackage{amsfonts}
\usepackage{amssymb}
\usepackage{enumerate}
\usepackage{setspace}
\usepackage{color}
\usepackage{esint}
\usepackage{dsfont}

\usepackage{tabu}
\usepackage{tikz} 
\usepackage{setspace}
\usepackage[margin=2.5cm]{geometry}

\usepackage[english]{babel}
\usepackage{graphicx}

\definecolor{luh-dark-blue}{rgb}{0.0, 0.313, 0.608}

\usepackage[colorlinks=true, citecolor=blue]{hyperref}

\newtheorem{theorem}{Theorem}[section]
\newtheorem{lemma}{Lemma}[section]

\newtheorem{corollary}{Corollary}[section]

\newtheorem{remark}{Remark}[section]

\newcommand{\eqn}{\begin{eqnarray}}
\newcommand{\een}{\end{eqnarray}}

\numberwithin{equation}{section}

\DeclareMathOperator{\dv}{div}

\newcommand{\R}{\mathbb R}

\newcommand{\bq}{\begin{equation}}
\newcommand{\eq}{\end{equation}}

\newcommand{\lt}{\left}
\newcommand{\rt}{\right}

\newcommand{\mx}{\mathcal X}
\newcommand{\my}{\mathcal Y}

\makeatletter
\def\moverlay{\mathpalette\mov@rlay}
\def\mov@rlay#1#2{\leavevmode\vtop{%
   \baselineskip\z@skip \lineskiplimit-\maxdimen
   \ialign{\hfil$\m@th#1##$\hfil\cr#2\crcr}}}
\newcommand{\charfusion}[3][\mathord]{
    #1{\ifx#1\mathop\vphantom{#2}\fi
        \mathpalette\mov@rlay{#2\cr#3}
      }
    \ifx#1\mathop\expandafter\displaylimits\fi}
\makeatother

\newcommand{\intr}{\int_{\R^3}}

 \makeatletter
 \newcommand{\norm}{\@ifstar{\@normb}{\@normi}}
 \newcommand{\@normb}[2]{\left\Vert{#1}\right\Vert_{#2}}
 \newcommand{\@normi}[2]{\Vert{#1}\Vert_{#2}}
 \makeatother

\newcommand\wh[1]{%
\savestack{\tmpbox}{\stretchto{%
  \scaleto{%
    \scalerel*[\widthof{\ensuremath{#1}}]{\kern-.6pt\bigwedge\kern-.6pt}%
    {\rule[-\textheight/2]{1ex}{\textheight}}
  }{\textheight}%
}{0.5ex}}%
\stackon[1pt]{#1}{\tmpbox}%
}

\usepackage{soul}

\begin{document}

\title[Navier-Stokes equations and Hall-MHD]{On the temporal estimates for the incompressible Navier-Stokes equations and the Hall-magnetohydrodynamic equations}

\author[Bae]{Hantaek Bae}
\address[Hantaek Bae]{\newline Department of Mathematical Sciences, \newline
Ulsan National Institute of Science and Technology (UNIST), Republic of Korea}
\email{hantaek@unist.ac.kr}

\author[Jung]{Jinwook Jung}
\address[Jinwook Jung]{\newline Department of Mathematics  and Research Institute for Natural Sciences \newline Hanyang University,  Republic of Korea}
\email{jinwookjung@hanyang.ac.kr}

\author[Shin]{Jaeyong Shin}
\address[Jaeyong Shin]{\newline Department of Mathematics \newline
Yonsei University, Republic of Korea}
\email{sinjaey@yonsei.ac.kr}

\date{\today}

\begin{abstract}
In this paper, we derive decay rates for solutions to the incompressible Navier-Stokes equations and Hall-magnetohydrodynamic equations. We first improve the decay rate of weak solutions to these equations by refining the Fourier splitting method with initial data in the space of pseudo-measures. Additionally, we investigate these equations with initial data in the Lei-Lin spaces and establish decay rates for those solutions.
\end{abstract}

\maketitle

\vspace{-5ex}

\section{Introduction}
In this paper, we study the large-time behavior of two parabolic systems in $\mathbb{R}^{3}$: the incompressible Navier-Stokes equations and the Hall-magnetohydrodynamic (Hall-MHD) equations. We begin with the incompressible Navier-Stokes equations and provide several decay rate results under appropriate assumptions on the initial data. We then turn our attention to the Hall-MHD equations and perform a similar analysis.

Before proceeding, we establish some notation. We let $C_{0}$ denote a generic constant depending on the various norms of the initial data and the parameters in our results, but independent of time. Moreover, $f\in L^\infty([0,\infty);X)$ signifies that $\displaystyle \sup_{0\leq t<\infty}\|f(t)\|_{X}\leq C_{0}$, where $X$ is a Banach space.

\subsection{The incompressible Navier-Stokes equations}
The incompressible Navier-Stokes equations are given by 
\eqn \label{NSE}
\begin{split}
&u_{t}  +u\cdot \nabla u +\nabla p-\mu\Delta u=0,\\
& \dv u=0,
\end{split}
\een
where $u$ is the fluid velocity and $p$ is the pressure. $\mu>0$ is a viscosity coefficient for which we set $\mu=1$ for simplicity. We begin with a weak solution of (\ref{NSE}) in $\mathbb{R}^{3}$. The existence of a global-in-time weak solution with a divergence-free initial datum $u_{0}\in L^{2}$ is proved in \cite{Leray} where $u$ satisfies the energy inequality 
\eqn\label{NSE energy}
\norm{u(t)}{L^{2}}^{2}+2\int^{t}_{0}\norm{\nabla u(\tau)}{L^{2}}^{2}\,d\tau\leq \norm{u_{0}}{L^{2}}^{2}
\een
for all $t>0$. We also notice that the following question is raised in \cite{Leray}: is $\|u(t)\|_{L^{2}}$ convergent to 0 as $t\rightarrow \infty$? This question is answered in \cite{Masuda}. Later, the decay rate of a weak solution is obtained in \cite{Schonbek} by using the Fourier splitting method with $u_{0}\in L^{2}\cap L^{1}$: $\|u(t)\|^2_{L^{2}}\leq C_{0}(1+t)^{-\frac{1}{2}}$. This is improved in \cite{Schonbek 2}: if $u_{0}\in L^{2}\cap L^{1}$,
\eqn \label{Decay Schonbek}
\|u(t)\|^{2}_{L^{2}}\leq C_{0}(1+t)^{-\frac{3}{2}}.
\een

We now define an invariant space of (\ref{NSE}). We say $\mathbb{X}$ is an invariant space when $u_{0} \in \mathbb{X}$ implies $u\in L^{\infty}([0,\infty);\mathbb{X})$. Since  $u(t)$ is also in $L^{1}$ for all $t>0$, which is proved in Appendix \ref{A 1}, $\mathbb{X}=L^{2}\cap L^{1}$ is an invariant space of (\ref{NSE}). In this paper, we seek  to find more invariant spaces which we employ to improve (\ref{Decay Schonbek}). When $u_{0}\in L^{2}\cap L^{1}$, 
\[
\text{\cite{Schonbek}}: \ \sup_{\xi \in S(t)}|\xi|\left|\widehat{u}(t,\xi)\right|\leq C_{0}, \quad \text{\cite{Schonbek 2}}: \ \sup_{\xi \in S(t)}\left|\widehat{u}(t,\xi)\right|\leq C_{0},\quad S(t):=\lt\{\xi\in\R^3 : |\xi|^2\le \frac{3}{2(1+t)}\rt\}, 
\]
where $C_{0}$ only depends on $u_{0}\in L^{2}\cap L^{1}$. So, to improve (\ref{Decay Schonbek}), it is natural to impose initial data in the space of pseudo-measures \cite{Cannone}:  
\[
\my^\sigma  := \Big\{ f \in \mathcal{S}'(\R^3) \ : \ \sup_{\xi \in \R^3} |\xi|^\sigma |\widehat f(\xi)| <\infty\Big\}.
\]

\begin{theorem} \label{Theorem 1} \upshape
Let $u_{0}\in L^{2}\cap \mathcal{Y}^{\sigma}$ with $\dv u_{0}=0$ and $\sigma \in [-1,1]$. Let $u$ be  a weak solution of (\ref{NSE}) satisfying (\ref{NSE energy}). Then, $u\in L^\infty([0,\infty);L^{2}\cap \my^{\sigma})$ and  $u$ decays in time as follows:
\eqn \label{Theorem 1 Decay}
\|u(t)\|_{L^2}^2  \le C_{0}(1+t)^{-\frac 32 + \sigma} \quad \text{for all $t>0$.}
\een
\end{theorem}

We have the same decay rate of (\ref{Decay Schonbek}) when $\sigma=0$, but $L^{1}\subset \my^0$. Moreover, we improve (\ref{Decay Schonbek}) because $\sigma$ can be negative. Theorem \ref{Theorem 1} can be proved by using the approach in \cite{Wiegner}: see Appendix \ref{A 2}. However, we are not able to  use the same approach to Hall-MHD equations: instead, we take another method that treat (\ref{NSE}) and Hall-MHD equations in the same way.

When $u_{0}\in L^{2}\cap L^{1}$, we already mention that $\|u(t)\|_{L^{1}}\leq C_{0}$ for all $t>0$. By combining this with (\ref{Decay Schonbek}), the decay rate of $L^{p}$ norms for $p\in (1,2)$ is given by 
\eqn \label{decay of Lp norm}
\|u(t)\|_{L^{p}}\leq \|u(t)\|^{\frac{2}{p}-1}_{L^{1}}\|u(t)\|^{2(1-\frac{1}{p})}_{L^{2}}\leq C_{0}\|u(t)\|^{2(1-\frac{1}{p})}_{L^{2}}\leq C_{0}(1+t)^{-\frac{3}{2}(1-\frac{1}{p})}.
\een
Although we do not have the uniform $L^{1}$ bound of $u$ when $u_{0}\in L^{2}\cap \mathcal{Y}^{\sigma}$, we utilize the decay rate in Theorem \ref{Theorem 1} to bound $u$ in the negative Sobolev spaces.

\begin{corollary} \label{Corollary 1}\upshape
If $\sigma \in [-1,1]$ and $0<\delta<\frac{3}{2}-\sigma$,
\eqn \label{decay of negative norm}
\left\|u(t)\right\|^{2}_{\dot{H}^{-\delta}} \leq C_{0}(1+t)^{-\frac{3}{2}+\sigma+\delta}.
\een
\end{corollary}

In particular, \eqref{decay of Lp norm} with $p\in(1,2)$ and \eqref{decay of negative norm} with $\sigma=0$ and $\delta\in(0,\frac{3}{2})$ have the same decay rate since $L^{p}\subset \dot{H}^{-\delta}$ with $1-\frac{1}{p}=\frac{1}{2}-\frac{\delta}{3}$. In \cite{Schonbek Wiegner}, the decay of higher-order norms for solutions is derived using the eventual regularity of weak solutions to \eqref{NSE} and the decay rate of $\|u(t)\|_{L^2}$. As a corollary of Theorem \ref{Theorem 1} and \cite{Schonbek Wiegner}, we can establish the following decay rates; the proof is omitted for brevity.

\begin{corollary}\label{Corollary 2} \upshape
Let $u_{0}\in L^{2}\cap \mathcal{Y}^{\sigma}$ with $\dv u_{0}=0$ and $\sigma \in [-1,1]$. For each $k\in \mathbb{N}$, there exist $T_{0}=T_0(u_0)>0$ and $C_{k}=C_k(u_0,k)>0$ such that 
\[
\left\|\nabla^{k}u(t)\right\|^{2}_{L^{2}}\leq C_{k} (1+t)^{-\frac 32 + \sigma-k } \quad \text{for all $t> T_{0}$}.
\]
\end{corollary}

If we take $u_{0}$ in a function space other than $L^{2}$, we normally impose a smallness condition on $u_{0}$ derived from its scaling invariance: $u_{0}(x) \longmapsto \lambda u_{0}(\lambda x)$. Along these lines, the best result is that of \cite{Koch Tataru}, which employs initial data in $\text{BMO}^{-1}$. However, since we aim to restrict the function spaces to those defined in Fourier variables, we investigate \eqref{NSE} with initial data in the Lei-Lin spaces $\mx^\sigma$:
\[
\mx^\sigma := \Big\{ f \in \mathcal{S}'(\R^3) \ : \ \intr |\xi|^\sigma |\widehat f(\xi)|\,d\xi <\infty\Big\}, \quad \sigma\in\R.
\]
In \cite{Lei Lin}, the global well-posedness of (\ref{NSE}) is established with small initial data $u_0 \in \mathcal{X}^{-1}$, where $\mathcal{X}^{-1} \subset \text{BMO}^{-1}$. Moreover, the spatial analyticity of the solution in \cite{Lei Lin} is established in \cite{Ambrose, Bae}, providing decay rates for the derivatives of $u$. The decay of $\|u(t)\|_{\mathcal{X}^{-1}}$ when $u_0 \in \mathcal{X}^{-1} \cap L^2$ is also studied in \cite{Benameur}. In this paper, we provide a decay rate for $\|u(t)\|_{\mathcal{X}^{-1}}$ under the condition that $\|u_0\|_{\mathcal{X}^{-1}}$ is sufficiently small and $u_0 \in \mathcal{Y}^\sigma$. More generally, we derive a decay rate for $\|u(t)\|_{\mathcal{X}^{k-1}}$ for any $k \geq 0$.

\begin{theorem} \label{Theorem 2} \upshape
Let $k\geq0$ and $\sigma \in [-1,1]$. Let $u_0 \in \mx^{k-1}\cap \mx^{-1}\cap \my^\sigma$ with $\dv u_0=0$. There exists $\epsilon>0$ such that if $\|u_0\|_{\mx^{-1}} \leq \epsilon$, then (\ref{NSE}) admits a unique  solution $u \in L^{\infty}([0,\infty); \mx^{k-1} \cap \mx^{-1}\cap \my^{\sigma}) \cap L^1((0,\infty); \mx^{k+1} \cap \mx^1)$. Furthermore, 
\[
\begin{split}
&\|u(t)\|_{\mx^{k-1}} \leq C_{k}(1+t)^{-1+\frac{\sigma}{2}-\frac{k}{2}} \quad \text{for all $t>0$.}
\end{split}
\]
\end{theorem}

The smallness condition in \cite{Lei Lin} is given by $\|u_0\|_{\mx^{-1}}<1$ (assuming $\mu=1$). However, for the proof of Theorem \ref{Theorem 2}, we require $\|u_0\|_{\mx^{-1}} \leq \epsilon$, where $\epsilon > 0$ may be smaller than 1.The case $k=0$ shows the decay rate of the solution in \cite{Lei Lin}, and our work extends this result to include the case $k>0$. A similar argument is applied to Theorem \ref{Theorem 5} below.

\subsection{Hall-magnetohydrodynamic equations}
We observe that the methodology developed for \eqref{NSE} can be applied to the incompressible and resistive Hall-MHD equations:
\begin{subequations}\label{Hall MHD}
\begin{align}
& u_{t}  +u\cdot \nabla u +\nabla p-\mu\Delta u=(\nabla \times B)\times B, \label{Hall MHD a} \\
&B_{t}  -\nabla \times (u\times B) -\nu\Delta B +\nabla \times ((\nabla \times B)\times B)=0,  \label{Hall MHD b} \\
& \dv u=0, \quad \dv B=0,
\end{align}
\end{subequations}
where $u$ is the plasma velocity field, $p$ is the pressure, and $B$ is the magnetic field, respectively. $\mu,\nu>0$ are viscosity and resistivity coefficients, and we set $\mu=\nu=1$ without loss of generality. The term $\nabla \times ((\nabla \times B) \times B)$ is referred to as the Hall term. The Hall-MHD system \eqref{Hall MHD} plays a crucial role in describing a wide range of physical phenomena \cite{Balbus, Forbes, Homann, Lighthill, Mininni, Shalybkov, Shay, Wardle}. Moreover, extensive mathematical studies have been conducted on \eqref{Hall MHD} following the seminal works of \cite{Acheritogaray, Chae-Degond-Liu, Chae Lee, Chae Schonbek}; for a comprehensive list of known results, we refer the reader to \cite{BK}.

 We begin with a weak solution of (\ref{Hall MHD}) with divergence-free initial data $(u_{0}, B_{0})\in L^{2}$. The existence of a weak solution satisfying the following energy inequality  is proved in \cite{Chae-Degond-Liu}:
\eqn \label{HMHD energy}
\norm{u(t)}{L^{2}}^{2}+\norm{B(t)}{L^{2}}^{2}+2\int^{t}_{0}\norm{\nabla u(\tau)}{L^{2}}^{2}\,d\tau+2\int^{t}_{0}\norm{\nabla B(\tau)}{L^{2}}^{2}\,d\tau\leq \norm{u_{0}}{L^{2}}^{2}+\norm{B_{0}}{L^{2}}^{2} \quad \text{for all $t>0$. }
\een  
Moreover, the decay of a weak solution is established in \cite{Chae Schonbek} using the Fourier splitting method: if $(u_{0}, B_{0})\in L^{2}\cap L^{1}$, then $(u,B)$ decays in time as follows:
\eqn \label{Hall MHD Weak decay}
\|u(t)\|^{2}_{L^{2}}+\|B(t)\|^{2}_{L^{2}}\leq C_{0}(1+t)^{-\frac{3}{2}}.
\een

We note that $L^1$ remains an invariant space for $u$, which can be established as in Appendix \ref{A 1} using \eqref{HMHD energy} and \eqref{Hall MHD Weak decay}. However, the presence of the Hall term prevents us from proving that $L^1$ is an invariant space for $B$. Analogous to the case of \eqref{NSE}, we present a result that improves upon \eqref{Hall MHD Weak decay}.

\begin{theorem}\label{Theorem 3} \upshape
Let $(u_{0}, B_{0})\in (L^{2}\cap \mathcal{Y}^{\sigma_{1}})\times (L^{2}\cap \mathcal{Y}^{\sigma_{2}})$ with $\dv u_{0}=\dv B_{0}=0$, $\sigma_1 \in [-1,1]$ and $\sigma_2 \in [-1,0]$. Let $(u,B)$ be a weak solution of (\ref{Hall MHD}) satisfying (\ref{HMHD energy}). Then, $u\in L^\infty([0,\infty);\my^{\sigma_1})$ and   
\[
\|u(t)\|_{L^2}^2+\|B(t)\|_{L^2}^2 \le C_{0}(1+t)^{-\frac 32 + \max{\{\sigma_1,\sigma_2\}} } \quad \text{for all $t>0$.}
\]
\end{theorem}

In addition to improving \eqref{Hall MHD Weak decay}, our result shows that $u_{0}$ and $B_{0}$ need not belong to the same space. Due to the Hall term, the range of $\sigma_{2}$ is required to be one less than that of $\sigma_{1}$. We also observe that, unlike $u$, we have $B \in L^\infty([0,T); \mathcal{Y}^{\sigma_2})$ for any finite $T > 0$; that is, $\|B(t)\|_{\mathcal{Y}^{\sigma_2}}$ may grow unbounded as $t \rightarrow \infty$. Nevertheless, we can still obtain a uniform-in-time bound on $\|B(t)\|_{\mathcal{Y}^{\sigma_2}}$, provided that
\begin{equation}\label{index_condition}
\sigma_1-\frac{\sigma_2}{2}\le 1,\ \sigma_2> -1,\quad\text{or}\quad \sigma_1<\frac{1}{2},\ \sigma_2=-1.
\end{equation}

\begin{corollary}\label{Corollary 3}\upshape
Under the assumptions in Theorem \ref{Theorem 3} with \eqref{index_condition}, $B\in L^\infty([0,\infty);\my^{\sigma_2})$.
\end{corollary}

Hence, if \eqref{index_condition} holds, $(L^{2}\cap \mathcal{Y}^{\sigma_{1}})\times (L^{2}\cap \mathcal{Y}^{\sigma_{2}})$ is an invariant space of $(u,B)$ for (\ref{Hall MHD}). A natural question arising from Theorem \ref{Theorem 3} is whether we can estimate $u$ and $B$ separately. To this end, higher-order regularity is necessary; therefore, we consider strong solutions to \eqref{Hall MHD}, whose global-in-time existence is guaranteed for sufficiently small initial data \cite{BKS, Chae-Degond-Liu, DL19, DT22, Tan}. This setting is also addressed in \cite{Chae Schonbek} to derive decay rates for higher-order norms of solutions based on \eqref{Hall MHD Weak decay}.

\begin{theorem}\label{Theorem 4}\upshape
Let $(u_{0}, B_{0})\in (L^{2}\cap \mathcal{Y}^{\sigma_{1}})\times (L^{2}\cap \mathcal{Y}^{\sigma_{2}})$ with $\dv u_{0}=\dv B_{0}=0$, $\sigma_1 \in [-1,1]$ and $\sigma_2 \in [-1,0]$. Additionally suppose that $(u_0,B_0)\in \dot{H}^{\frac{1}{2}}\times(\dot{H}^{\frac{1}{2}}\cap\dot{H}^{\frac{3}{2}})$ and $\|u_0\|_{\dot{H}^{\frac{1}{2}}}+\|B_0\|_{\dot{H}^{\frac{1}{2}}}+\|B_0\|_{\dot{H}^{\frac{3}{2}}}$ is sufficiently small. Then \eqref{Hall MHD} admits a unique solution $(u,B)\in L^\infty((0,\infty);(H^{\frac{1}{2}}\cap\my^{\sigma_1})\times(H^{\frac{3}{2}}\cap\my^{\sigma_2}))$. Moreover, for any $k\in\mathbb{N}\cup \{0\}$ there exists $T_k=T_k(u_0,B_0, k)>0$ and $C_k=C_k(u_0,B_0,k)>0$ such that
\[
\|\nabla^k u(t)\|_{L^2}^2\leq C_k(1+t)^{-\frac{3}{2}+\sigma_1-k},\quad \|\nabla^k B(t)\|_{L^2}^2\leq C_k(1+t)^{-\frac{3}{2}+\sigma_2-k}\quad\text{for all $t>T_k$.}
\]
\end{theorem}

\begin{remark}\upshape
The additional condition on the initial data $(u_0,B_0)$ in the critical Sobolev space $\dot{H}^{\frac{1}{2}}\times(\dot{H}^{\frac{1}{2}}\cap\dot{H}^{\frac{3}{2}})$ is required only for the global existence of strong solutions, as established in \cite{DT22, Tan}. In fact, the eventual regularity of weak solutions to \eqref{Hall MHD} suffices to derive the desired decay rates in Theorem \ref{Theorem 4}. However, unlike the case of weak solutions to \eqref{NSE}, this property remains an open problem for weak solutions to \eqref{Hall MHD}. This issue will be addressed in a forthcoming work \cite{JS-arXiv} by the second and third authors.
\end{remark}

We also deal with (\ref{Hall MHD}) in Lei-Lin spaces $\mx^\sigma$. The global well-posedness of (\ref{Hall MHD})  is  established when $(u_{0}, B_{0})$ is sufficiently small in $\mathcal{X}^{-1}\cap \mathcal{X}^{0}$ \cite{Kwak} and $(u_{0}, B_{0}, \nabla \times B_{0})$ is  sufficiently small in $\mathcal{X}^{-1}$ \cite{Liu}. We here establish the global well-posendess and derive the decay rate of $(u,B)$ in Lei-Lin spaces using $\my^\sigma$.

\begin{theorem}\label{Theorem 5} \upshape
Let $k\geq 0$ and $\sigma \in [-1,1]$. Let $u_0 \in \mx^{k-1}\cap \mx^{-1}\cap \my^\sigma$ and $B_0 \in \mx^{k-1}\cap \mx^{-1}\cap\mx^0 \cap \my^\sigma$ with $\dv u_0=\dv B_0=0$. There exists $\epsilon>0$ such that if $\|u_0\|_{\mx^{-1}} + \|B_0\|_{\mx^{-1}}+ \|B_0\|_{\mx^0}\leq \epsilon$, then (\ref{Hall MHD}) admits a unique  solution 
\[
\begin{split}
&u \in L^{\infty}([0,\infty); \mx^{k-1} \cap \mx^{-1}\cap \my^{\sigma}) \cap L^1((0,\infty); \mx^{k+1} \cap \mx^1), \\
& B \in L^{\infty}([0,\infty); \mx^{k-1} \cap \mx^{-1}\cap \mx^0 \cap \my^\sigma) \cap L^1((0,\infty); \mx^{k+1}\cap \mx^1\cap \mx^2).
\end{split}
\]
Furthermore, 
\[
\begin{split}
&\|u(t)\|_{\mx^{k-1}} + \|B(t)\|_{\mx^{k-1}}\leq C_k(1+t)^{-1+\frac{\sigma}{2}-\frac{k}{2}} \quad \text{for all $t>0$.}
\end{split}
\]
\end{theorem}

\section{ Preliminaries} \label{sec:Prelim}
All constants will be denoted by $C$, and we follow the convention that such constants may vary from line to line, or even between two occurrences within the same expression. We also use a simplified notation for integrals over the spatial variables:
\[
\int f=:\int_{\mathbb{R}^{3}}f(x)dx.
\]

We begin with 2 inequalities:  
\begin{enumerate}[]
\item (1) For all $x>0$ and $p>0$,
\begin{equation}\label{heat_inf_est}
|x^p e^{-ax^2}| \le  |x^p e^{-ax^2}| \bigg|_{x = \sqrt{\frac{p}{2a}}} = \lt(\frac{p}{2a} \rt)^{\frac p2} e^{-\frac p2}.
\end{equation}
\item (2) For $0<\alpha,\beta<1$ with $\alpha+\beta=1$
\[
\int^{t}_{0}(t-\tau)^{-\alpha}\tau^{-\beta}d\tau=\int_0^1 (1-\theta)^{-\alpha}\theta^{-\beta}\,d\theta = \mathcal{B}\lt(\alpha,\beta\rt),
\]
where $\mathcal{B}(\cdot, \cdot)$ is the beta function.
\end{enumerate}

We give the following Sobolev inequalities in 3D:
\[
\left\|f\right\|_{L^{3}}\leq C\|f\|^{\frac{1}{2}}_{L^2}\|\nabla f\|^{\frac{1}{2}}_{L^2}, \quad \left\|f\right\|_{L^{6}}\leq C\|\nabla f\|_{L^2},
\]
and a product estimate {which is called} Leibniz rule \cite{Christ}, \cite[Page 105]{Taylor}: for $1<p<\infty$ and $1<p_{1},p_{2},p_{3},p_{4}\leq\infty$
\eqn \label{fractional Leibniz}
\left\|\nabla^{k}(fg)\right\|_{L^{p}}\leq C \left\|\nabla^{k} f\right\|_{L^{p_{1}}}\left\|g\right\|_{L^{p_{2}}}+C\left\|f\right\|_{L^{p_{3}}}\left\|\nabla^{k}g\right\|_{L^{p_{4}}}, \quad \frac{1}{p_{1}}+\frac{1}{p_{2}}=\frac{1}{p_{3}}+\frac{1}{p_{4}}=\frac{1}{p}
\een

We finally present 3 interpolation results in $\mx^\sigma$:
\begin{subequations}
\begin{align}
\sigma\in (-1,1): &  \ \|f\|_{\mx^{-\sigma}} \le C\|f\|_{\mx^{-1}}^{\frac{1+\sigma}{2}} \|f\|_{\mx^1}^{\frac{1-\sigma}{2}}, \label{intp_x 1}\\
\sigma\in (-1,0): & \ \|f\|_{\mx^{-\sigma}} \le C\|f\|_{\mx^0}^{\frac{2+\sigma}{2}}\|f\|_{\mx^2}^{\frac{-\sigma}{2}}, \label{intp_x 2}\\
\sigma\in [0,k): & \ \|f\|_{\mx^{\sigma}} \le C\|f\|_{\mx^{-1}}^{\frac{k-\sigma}{k+1}}\|f\|_{\mx^k}^{\frac{\sigma+1}{k+1}}. \label{intp_x 3}
\end{align}
\end{subequations}

\subsection{Main Lemmas}
We now provide two lemmas being central to the proofs of our decay rate results.

\begin{lemma}\label{main lemma 1}\upshape
Let $f$ be a smooth function satisfying 
\[
\frac{d}{dt}\|f(t)\|_{L^{2}}^2 + \|\nabla f(t)\|_{L^{2}}^2 \le 0
\]
for all $t>0$. Suppose there exists a positive constant $C_{\ast}>0$ and $\sigma < \frac 32$ such that
\[
\sup_{0\leq t<\infty}\sup_{|\xi|\leq 1}|\xi|^\sigma \left|\widehat f(t,\xi)\right|  \le C_{\ast}.
\]
Then, 
\[
\|f(t)\|_{L^{2}}^2  \le C(1+t)^{-\frac{3}{2}+\sigma}\quad\text{for all $t\geq N-1$},
\]
where $N-1$ is a non-negative constant with $N>\frac{3}{2}-\sigma$.
\end{lemma}

\begin{proof}
By using the Plancherel's theorem, 
\[
\begin{aligned}
\frac{d}{dt}\|f(t)\|_{L^{2}}^2 & \le - \int |\xi|^{2} \left|\widehat f(t,\xi)\right|^2  \,d\xi\le -\int_{\{|\xi|^2 > \frac{N}{1+t} \}} |\xi|^{2} \left|\widehat f(t,\xi)\right|^2 \,d\xi\le -\frac{N}{1+t} \int_{\{|\xi|^2 > \frac{N}{1+t}\}} \left|\widehat f(t,\xi)\right|^2 \,d\xi\\
&= -\frac{N}{1+t}\|f(t)\|_{L^{2}}^2 + \frac{N}{1+t}\int_{\{|\xi|^2 \le \frac{N}{1+t}\}} \left|\widehat f(t,\xi)\right|^2  \,d\xi.
\end{aligned}
\]
Hence, we attain
\[
\begin{aligned}
\frac{d}{dt}\|f(t)\|_{L^{2}}^2 +\frac{N}{1+t}\| f(t)\|_{L^{2}}^2 \leq   \frac{N}{1+t}\int_{\{|\xi|^2 \le \frac{N}{1+t}\}}\left|\widehat f(t,\xi)\right|^2  \,d\xi,
\end{aligned}
\]
which implies
\eqn \label{eq:2.6}
\frac{d}{dt}\left[(1+t)^N\|f(t)\|_{L^{2}}^2  \right] \leq N(1+t)^{N-1}\int_{\{|\xi|^{2} \le \frac{N}{1+t}\}} \left|\widehat f(t,\xi)\right|^2  \,d\xi.
\een
Since $\left\{\xi: |\xi|^2 \le N/(1+t)\right\} \subseteq \left\{\xi: |\xi|\le 1\right\}$ for all $t\geq N-1$, we use our assumption to bound the right-hand side of (\ref{eq:2.6}) as follows:
\[
\begin{aligned}
\frac{d}{dt}\left[(1+t)^N\|f(t)\|_{L^{2}}^2  \right] &\le NC_\ast (1+t)^{N-1} \int_{\{|\xi|^2 \le \frac{N}{1+t}\}} |\xi|^{-2\sigma} \,d\xi\\
&\le C(1+t)^{N-1}\int_0^{\sqrt{N/(1+t)}} r^{2-2\sigma}\,dr\le C(1+t)^{N -\frac{5}{2}+\sigma }
\end{aligned}
\]
when $\sigma < \frac 32$. Finally, we integrate the inequality with respect to time and obtain
\[
\begin{split}
\|f(t)\|_{L^{2}}^2 &\leq \frac{\|f(N-1)\|_{L^{2}}^2 }{(1+t)^{N}}+\frac{C}{N-\frac{3}{2}+\sigma} \frac{(1+t)^{N-\frac{3}{2}+\sigma}-N^{N-\frac{3}{2}+\sigma}}{(1+t)^{N}} \leq C_{0}(1+t)^{-\frac{3}{2}+\sigma}
\end{split}
\]
for $t\geq N-1$, where we use $N>\frac{3}{2}-\sigma$. This complete the proof of Lemma \ref{main lemma 1}.
\end{proof}

\begin{lemma}\label{main lemma 2}\upshape
Let $k \ge 0$ and $\theta>0$. Let $f$ be a smooth function satisfying the following inequality
\eqn \label{X inequality}
\frac{d}{dt}\|f(t)\|_{\mx^{k-1}}  + \theta \|f(t)\|_{\mx^{k+1}} \le 0
\een
for all $t>0$. Suppose there exists a positive constant $C_\ast>0$  and {$\sigma < k+2$} such that 
\eqn \label{X inequality 2}
\sup_{0\leq t<\infty} \sup_{|\xi\leq 1} |\xi|^\sigma \left|\widehat f(t,\xi)\right|\le C_\ast.
\een
Then, 
\[
\|f(t)\|_{\mx^{k-1}} \le C_{0}(1+t)^{-1 + \frac{\sigma}{2}-\frac{k}{2}}\quad\text{for all $t\geq N/\theta-1$},
\]
where $N/\theta -1$ is a non-negative constant with $N>1+\frac{k}{2}-\frac{\sigma}{2}$.
\end{lemma}

\begin{proof}
From (\ref{X inequality}), 
\[
\frac{d}{dt}\|f(t)\|_{\mx^{k-1}}  \le -\theta \int_{\{ \theta|\xi|^2 > \frac{N}{1+t}\}} |\xi|^{k+1} \left|\widehat f(t,\xi)\right| \,d\xi \leq -\frac{N}{1+t}\|f(t)\|_{\mx^{k-1}} + \frac{N}{1+t}\int_{\{ \theta|\xi|^2 \le \frac{N}{1+t} \}} |\xi|^{k-1} \left|\widehat f(t,\xi)\right|\,d\xi.
\]
Since $\left\{\xi: \theta |\xi|^2 \le N/(1+t)\right\} \subseteq \left\{\xi: |\xi| \le 1\right\}$ for $t\ge N/\theta-1$, we obtain
\[
\begin{aligned}
\frac{d}{dt}\lt[ (1+t)^N \|f(t)\|_{\mx^{k-1}} \rt] &\le NC_\ast (1+t)^{N-1}\int_{\{\theta|\xi|^2 \le \frac{N}{1+t} \}} |\xi|^{k-\sigma-1}\,d\xi\\
&\le C(1+t)^{N-1}\int_{0}^{\sqrt{N/(\theta(1+t))}} r^{k-\sigma+1} \,dr \le C(1+t)^{N-2+\frac{\sigma}{2}-\frac{k}{2}}
\end{aligned}
\]
from this we deduce that   
\[
\begin{split}
\|f(t)\|_{\mx^{k-1}} &\leq \frac{\|f(N/\theta-1)\|_{\mx^{k-1}} }{(1+t)^{N}}+\frac{C}{N-1+\frac{\sigma}{2}-\frac{k}{2}} \frac{(1+t)^{N-1+\frac{\sigma}{2}-\frac{k}{2}}-(N/\theta)^{N-1+\frac{\sigma}{2}-\frac{k}{2}}}{(1+t)^{N}}\leq C_{0}(1+t)^{-1+\frac{\sigma}{2}-\frac{k}{2}}
\end{split}
\]
for $t\geq N/\theta-1$, where we use $N>1+\frac{k}{2}-\frac{\sigma}{2}$. This completes the proof of Lemma \ref{main lemma 2}.
\end{proof}

\section{Incompressible Naiver-Stokes equations}  \label{sec:NS}
In this section, we prove the decay rates results for (\ref{NSE}). Let 
\eqn \label{NSE Integral}
u(t)=e^{t\Delta}u_{0}-\int^{t}_{0}e^{(t-\tau)\Delta}\mathbb{P}\left(\dv (u\otimes u)\right)(\tau)\,d\tau,
\een
where $\mathbb{P}$ is the Leray projection operator with its matrix valued Fourier multiplier $m(\xi)$: 
\eqn \label{Leray projection}
m_{ij}(\xi)=\delta_{ij}-\xi_{i}\xi_{j}/|\xi|^{2}.
\een
By taking the Fourier transform to (\ref{NSE Integral}), we have
\eqn\label{Fourier NS Integral}
\widehat{u}(t,\xi)=e^{-t|\xi|^{2}}\widehat{u}_{0}(\xi)-\int^{t}_{0}e^{-(t-\tau)|\xi|^{2}} m(\xi)i\xi\cdot \left(\widehat{u\otimes u} \right)(\tau,\xi)d\tau.
\een
Since $|m_{ij}(\xi)|\leq 1$, we proceed to bound $u$ as if $m$ is absent. We also bound $\|\widehat{f g}\|_{L_\xi^\infty}$  by Young's inequality and the Plancherel's theorem: $\|\widehat{f g}\|_{L_\xi^\infty} \leq \|f\|_{L^{2}} \|g\|_{L^{2}}$.

\subsection{Proof of Theorem \ref{Theorem 1}} \label{sec:NSE}
Since we already have (\ref{NSE energy}), we only need to show $u\in L^\infty([0,\infty);\my^{\sigma})$. Then, Lemma \ref{main lemma 1} gives the desired decay rate in Theorem \ref{Theorem 1}. To prove $u\in L^\infty([0,\infty);\my^{\sigma})$, we divide  the range of $\sigma$ into  3 cases.

\vspace{1ex}

\noindent
$\blacktriangleright$ ({\bf Case 1}: $\sigma=1$): By multiplying (\ref{Fourier NS Integral}) by $|\xi|$,
\[
\begin{aligned}
|\xi| |\widehat u(t,\xi)| &\leq |\xi|e^{-t |\xi|^2} |\widehat u_0(\xi)| + \int_0^t |\xi|^2 e^{-(t-\tau)|\xi|^2} \|\widehat{u \otimes u}(\tau)\|_{L_\xi^\infty}\,d\tau\\
&\leq \|u_0\|_{\my^1} + \int_0^t |\xi|^2 e^{-(t-\tau)|\xi|^2}\|u(\tau)\|_{L^2}^2\,d\tau \leq \|u_0\|_{\my^1} + \|u_0\|_{L^2}^2 \int_0^t |\xi|^2 e^{-(t-\tau)|\xi|^2}\,d\tau \leq C_{0}
\end{aligned}
\]
from which we deduce that 
\[
\sup_{0\le t <\infty} \|u(t)\|_{\my^1} \leq C_{0}.
\]

\noindent 
$\blacktriangleright$ ({\bf Case 2}: $\sigma \in [0,1)$). For $|\xi|\leq1$, we obtain
\[
|\xi| |\widehat u(t,\xi)| \leq |\xi|^{1-\sigma}|\xi|^{\sigma}e^{-t |\xi|^2} |\widehat u_0(\xi)| + \int_0^t |\xi|^2 e^{-(t-\tau)|\xi|^2} \|u(\tau)\|_{L^2}^2\,d\tau \leq \|u_0\|_{\my^{\sigma}} + \|u_0\|_{L^2}^2\leq C_0.
\]
Then, Lemma \ref{main lemma 1} yields
\[
\|u(t)\|_{L^2}^2  \le C_{0}(1+t)^{-\frac 12} \quad \text{for all $t>0$.}  
\]
By \eqref{heat_inf_est},
\bq\label{NS_Y_est1}
\begin{aligned}
|\xi|^{\sigma} |\widehat u(t,\xi)| &\le  |\xi|^{\sigma} e^{-t |\xi|^2} |\widehat u_0(\xi)| + \int_0^t |\xi|^{1+\sigma} e^{-(t-\tau)|\xi|^2} \|u(\tau)\|_{L^2}^2\,d\tau\\
&\le \|u_0\|_{\my^{\sigma}} + C_{0}\int_0^t (t-\tau)^{-\frac{1+\sigma}{2}}(1+\tau)^{-\frac 12}\,d\tau\leq C_{0},
\end{aligned}
\eq
where we use
\eqn \label{time bound 1}
\begin{aligned}
\int_0^t (t-\tau)^{-\frac{1+\sigma}{2}}(1+\tau)^{-\frac 12}\,d\tau&\le \int_0^t (t-\tau)^{-\frac{1+\sigma}{2}}(1+\tau)^{-\frac {1-\sigma}{2}}\,d\tau\le \int_0^t (t-\tau)^{-\frac{1+\sigma}{2}}\tau^{-\frac {1-\sigma}{2}}\,d\tau\\
&= \int_0^1 (1-\theta)^{-\frac{1+\sigma}{2}}\theta^{-\frac {1-\sigma}{2}}\,d\theta = \mathcal{B}\lt(\frac{1-\sigma}{2}, \frac{1+\sigma}{2} \rt).
\end{aligned}
\een

\noindent 
$\blacktriangleright$ ({\bf Case 3}: $\sigma \in [-1,0]$)
When $\sigma\in[-1,0]$, we have
\[
\sup_{0\leq t<\infty}\sup_{|\xi|\leq1}|\xi||\widehat{u}(t,\xi)|\leq C_0,\quad\text{and}\quad \norm{u(t)}{L^2}^2\leq C_0(1+t)^{-\frac{1}{2}}
\]
for all $t>0$ as in \textbf{Case 2}. For $|\xi|\leq1$ we obtain
\[
\begin{aligned}
|\widehat u(t,\xi)| &\le e^{-t |\xi|^2} |\widehat u_0(\xi)| + \int_0^t |\xi| e^{-(t-\tau)|\xi|^2} \|u(\tau)\|_{L^2}^2\,d\tau \le \|u_0\|_{\my^{\sigma}} + C_{0}\int_0^t (t-\tau)^{-\frac{1}{2}}(1+\tau)^{-\frac 12}\,d\tau\leq C_{0}
\end{aligned}
\]
which implies 
\[
\|u(t)\|_{L^2}^2  \le C(1+t)^{-\frac 32} \quad \text{for all $t>0$}.
\]
So, we bound $u$ as follows:
\bq\label{NS_Y_est2}
\begin{aligned}
|\xi|^{\sigma} |\widehat u(t,\xi)| &\le |\xi|^{\sigma}e^{-t |\xi|^2} |\widehat u_0(\xi)| + \int_0^t |\xi|^{1+\sigma} e^{-(t-\tau)|\xi|^2}\|u(\tau)\|_{L^2}^2\,d\tau\\
&\le \|u_0\|_{\my^{\sigma}} + C_{0}\int_0^t (t-\tau)^{-\frac{1+\sigma}{2}}(1+\tau)^{-\frac 32} \,d\tau \leq C_{0}.
\end{aligned}
\eq
Thus, we use Lemma \ref{main lemma 1} to obtain the desired temporal decay and this completes the proof of Theorem \ref{Theorem 1}.

\subsection{Proof of Corollary \ref{Corollary 1}}
We decompose $\left\|u\right\|^{2}_{\dot{H}^{-\delta}}$ as follows: 
\[
\left\|u\right\|^{2}_{\dot{H}^{-\delta}}=\int_{|\xi|\leq M} |\xi|^{-2\delta} \left|\widehat{u}(\xi)\right|^{2}d\xi+\int_{|\xi|\ge M} |\xi|^{-2\delta} \left|\widehat{u}(\xi)\right|^{2}d\xi\leq \text{(I)}+M^{-2\delta}\|u\|^{2}_{L^{2}},
\]
where $M$ is decided below. We now bound $\text{(I)}$: when $\delta<\frac{3}{2}-\sigma$,
\[
\text{(I)}=\int_{|\xi|\leq M} |\xi|^{-2\delta-2\sigma} |\xi|^{2\sigma}\left|\widehat{u}(\xi)\right|^{2}d\xi\leq C\|u\|^{2}_{\my^{\sigma}}\int^{M}_{0}r^{-2\delta-2\sigma+2}dr=C \|u\|^{2}_{\my^{\sigma}} M^{-2\delta-2\sigma+3}.
\]
By taking $M$ satisfying $M^{-2\delta}\|u\|^{2}_{L^{2}}=\|u\|^{2}_{\my^{\sigma}} M^{-2\delta-2\sigma+3}$, 
we have
\[
\left\|u(t)\right\|^{2}_{\dot{H}^{-\delta}} \leq C \|u(t)\|^{\frac{2\delta}{3/2-\sigma}}_{\my^{\sigma}} \|u(t)\|^{\frac{3-2\sigma-2\delta}{3/2-\sigma}}_{L^{2}}\leq C_{0}(1+t)^{-\frac{3}{2}+\sigma+\delta}
\]
which implies the desired estimate.

\subsection{Proof of Theorem \ref{Theorem 2}}
From \cite{Lei Lin}, we deduce that if $\norm{u_0}{\mx^{-1}}\leq\epsilon<1$, then $u\in L^{\infty}([0,\infty);\mx^{-1})\cap L^{1}([0,\infty);\mx^{1})$ and 
\[
\norm{u(t)}{\mx^{-1}}+(1-\epsilon)\int^{t}_{0}\norm{u(\tau)}{\mx^{1}}\,d\tau\leq\epsilon\quad\text{for all $t>0$}.
\]
Moreover, $u$ satisfies \eqref{X inequality} with $k=0$ and $\theta=1-\epsilon$. We now show  $u\in L^\infty([0,\infty);\my^{\sigma})$. To do so, we use    
\eqn  \label{est_conv 1}
1 \le \frac{|\eta|^\alpha}{2|\xi-\eta|^\alpha} + \frac{|\xi-\eta|^\alpha}{2|\eta|^\alpha}.
\een

\noindent $\blacktriangleright$ ({\bf Case 1}: $\sigma=1$) By using (\ref{est_conv 1}) with $\alpha=1$, we obtain 
\[
\begin{aligned}
|\xi||\widehat u(t,\xi)| &\le |\xi| e^{-t|\xi|^2 }|\widehat u_0(\xi)| + \int_0^t |\xi|^2  e^{-(t-\tau)|\xi|^2}\int |\widehat u(\tau,\xi-\eta)|  |\widehat u(\tau,\eta)|\,d\eta d\tau\\
&\le \|u_0\|_{\my^1} + \sup_{0 \le \tau \le t} \left(\|u(\tau)\|_{\my^1} \|u(\tau)\|_{\mx^{-1}} \right)\int_0^t  |\xi|^2 e^{-(t-\tau)|\xi|^2}\,d\tau\leq \|u_0\|_{\my^1}+ C\epsilon \sup_{0 \le \tau \le t}\|u(\tau)\|_{\my^1}
\end{aligned}
\]
and so we have 
\[
\sup_{0 \le \tau \le t}\|u(\tau)\|_{\my^1} \leq \|u_0\|_{\my^1} +C\epsilon \sup_{0 \le \tau \le t}\|u(\tau)\|_{\my^1}.
\]

\noindent $\blacktriangleright$ ({\bf Case 2}: $\sigma \in [-1,1)$) When $\sigma \in (-1,1)$, we use  (\ref{est_conv 1}) with $\alpha=\sigma$ to estimate $u$ as
\[
\begin{aligned}
|\xi|^\sigma |\widehat u(t,\xi)| &\le |\xi|^\sigma e^{-t|\xi|^2}|\widehat u_0(\xi)|+ \int_0^t |\xi|^{1+\sigma}  e^{-(t-\tau)|\xi|^2}\intr |\widehat u(\tau,\xi-\eta)|  |\widehat u(\tau,\eta)|\,d\eta d\tau\\
&\le \|u_0\|_{\my^\sigma} +\int_0^t |\xi|^{1+\sigma}  e^{-(t-\tau)|\xi|^2} \|u(\tau)\|_{\mx^{-\sigma}}\|u(\tau)\|_{\my^\sigma}\, d\tau\\
&\le \|u_0\|_{\my^{\sigma}} +\sup_{0\le \tau \le t}\|u(\tau)\|_{\my^\sigma} \Big(\int_0^t |\xi|^2 e^{-(t-\tau)\frac{2}{1+\sigma}|\xi|^2 }\,d\tau \Big)^{\frac{1+\sigma}{2}}\Big(\int_0^t \|u(\tau)\|_{\mx^{-\sigma}}^{\frac{2}{1-\sigma}}\,d\tau \Big)^{\frac{1-\sigma}{2}}\\
&\le \|u_0\|_{\my^\sigma} + C\sup_{0\le \tau \le t}\|u(\tau)\|_{\my^\sigma} \Big(\sup_{0\le \tau \le t}\|u(\tau)\|_{\mx^{-1}} \Big)^{\frac{1+\sigma}{2}} \Big(\int_0^t \|u(\tau)\|_{\mx^1} \,d\tau\Big)^{\frac{1-\sigma}{2}}\\
&\leq \|u_0\|_{\my^\sigma}+ C\epsilon \sup_{0 \le \tau \le t}\|u(\tau)\|_{\my^\sigma}, 
\end{aligned}
\]
where we also use (\ref{intp_x 1}) to handle $\|u(\tau)\|_{\mx^{-\sigma}}$. 

When $\sigma=-1$, we use \eqref{est_conv 1} with $\alpha=1$ to get
\[
\begin{aligned}
|\xi|^{-1} |\widehat u(t,\xi)| &\le |\xi|^{-1} e^{-t|\xi|^2}|\widehat u_0(\xi)|+ \int_0^t  e^{-(t-\tau)|\xi|^2}\int |\widehat u(\tau,\xi-\eta)|  |\widehat u(\tau,\eta)|\,d\eta d\tau\\
&\le \|u_0\|_{\my^{-1}} + \sup_{0\le \tau \le t}\|u(\tau)\|_{\my^{-1}} \int_0^t \|u(\tau)\|_{\mx^1} \,d\tau.
\end{aligned}
\]

\noindent $\blacktriangleright$ Combining these estimates, we arrive at the following inequality
\[
\sup_{0 \le t < \infty} \|u(t)\|_{\my^{\sigma}} \leq \|u_0\|_{\my^{\sigma}} +C\epsilon \sup_{0 \le t < \infty}\|u(t)\|_{\my^{\sigma}}
\]
for all $\sigma\in [-1,1]$. By restricting the size of $\epsilon>0$ as $2C\epsilon< 1$, we obtain $u \in L^{\infty}([0,\infty); \my^{\sigma})$. The desired decay rates follow by Lemma \ref{main lemma 2}.

\vspace{1ex}

\noindent $\blacktriangleright$ Finally when $k>0$, we need to show that $u\in L^{\infty}([0,\infty);\mx^{k-1})\cap L^{1}([0,\infty);\mx^{k+1})$ and $u$ satisfies \eqref{X inequality} with $k>0$. By using \eqref{intp_x 3}, we obtain 
\[
\begin{aligned}
\frac{d}{dt}\|u(t)\|_{\mx^{k-1}} + \|u(t)\|_{\mx^{k+1}} &\le \int\int |\xi|^{k}|\widehat u (t, \xi-\eta)| |\widehat u(t,\eta)| \,d\eta d\xi \le C\|u(t)\|_{\mx^{0}}\|u(t)\|_{\mx^k}\\
&\le C\|u(t)\|_{\mx^{-1}}\|u(t)\|_{\mx^{k+1}} \leq C\epsilon \|u(t)\|_{\mx^{k+1}}. 
\end{aligned}
\]
Then, we can choose $\epsilon > 0$ sufficiently small such that $\theta = 1 - C\epsilon > 0$. This yields \eqref{X inequality} with $k > 0$, and a direct application of Lemma \ref{main lemma 2} completes the proof of Theorem \ref{Theorem 2}.

\section{ Hall-magnetohydrodynamic equations} \label{sec:HMHD}
In this section, we prove the decay rates results for (\ref{Hall MHD}). We first write 
\[
(\nabla\times B)\times B=\dv(B\otimes B)-\frac{1}{2}\nabla |B|^{2},\quad \nabla \times ((\nabla \times B)\times B)=\nabla \times \dv (B\otimes B),
\]
and express $(u,B)$ as the integral form
\begin{subequations}\label{HMHD Integral}
\begin{align}
u(t)&=e^{t\Delta}u_{0}-\int^{t}_{0}e^{(t-\tau)\Delta}\mathbb{P}\left(\dv (u\otimes u) -\dv(B\otimes B) \right)(\tau)\,d\tau, \label{HMHD Integral a}\\
B(t)&=e^{t\Delta}B_{0}+\int^{t}_{0}e^{(t-\tau)\Delta}\left(\nabla \times (u\times B) -\nabla \times \dv (B\otimes B)\right)(\tau)\,d\tau. \label{HMHD Integral b}
\end{align}
\end{subequations}
By taking the Fourier transform to (\ref{HMHD Integral}), we have
\[
\begin{split}
\widehat{u}(t,\xi)&=e^{-t|\xi|^{2}}\widehat{u}_{0}(\xi)-\int^{t}_{0}e^{-(t-\tau)|\xi|^{2}} m(\xi)i\xi\cdot \left(\widehat{u\otimes u} - \widehat{B\otimes B}\right)(\tau,\xi)d\tau,\\
\widehat{B}(t,\xi)&=e^{-t|\xi|^{2}}\widehat{B}_{0}(\xi) + \int^{t}_{0}e^{-(t-\tau)|\xi|^{2}}i\xi \times (\widehat{u\times B})(\tau, \xi)d\xi - \int^{t}_{0}e^{-(t-\tau)|\xi|^{2}} i\xi \times i \xi \cdot (\widehat{B\otimes B})(\tau,\xi)d\xi.
\end{split}
\]

\subsection{Proof of Theorem \ref{Theorem 3}}
Although \eqref{HMHD Integral a} contains the term $B \otimes B$, we can utilize \eqref{HMHD energy} to bound $u$ analogously to the case of \eqref{NSE}. Thus, we apply the estimates for $u$ derived in Section \ref{sec:NSE} and focus on the estimation of $B$. Throughout the proof of Theorem \ref{Theorem 3}, we will repeatedly invoke Lemma \ref{main lemma 1} with $\|f\|^{2}_{L^{2}} = \|u\|_{L^2}^2 + \|B\|_{L^2}^2$ and $\|\nabla f\|^{2}_{L^{2}} = \|\nabla u\|_{L^2}^2 + \|\nabla B\|_{L^2}^2$. We divide the range of $(\sigma_1, \sigma_2)$ into two cases.

\vspace{1ex}

\noindent $\blacktriangleright$ ({\bf Case 1}: $\sigma_1\in[0,1]$, $\sigma_2\in[-1,0]$) Using \eqref{HMHD energy} we first have
\eqn \label{u Y1}
\sup_{0\le t <\infty}\sup_{|\xi|\leq1} |\xi||\widehat u(t,\xi)| \leq C_{0}.
\een
For $B$, we note that for $|\xi|\le 1$,
\[
\begin{aligned}
|\xi| |\widehat B(t,\xi)| &\le |\xi| e^{-t|\xi|^2}|\widehat B_0(\xi)| + \int_0^t |\xi|^2 e^{-(t-\tau)|\xi|^2} \|u(\tau)\|_{L^2}\|B(\tau)\|_{L^2}\,d\tau + \int_0^t |\xi|^3 e^{-(t-\tau)|\xi|^2}\|B(\tau)\|_{L^2}^2\,d\tau\\
&\le \|B_0\|_{\my^{\sigma_2}} + 2(\|u_0\|_{L^2}^2+\|B_0\|_{L^2}^2)\int_0^t |\xi|^2 e^{-(t-\tau)|\xi|^2} \,d\tau \le  C_{0}.
\end{aligned}
\]
By combining this with (\ref{u Y1}), 
\[
\sup_{0\le t<\infty} \sup_{|\xi| \le  1} \lt(|\xi| |\widehat u(t,\xi)| + |\xi||\widehat B(t,\xi)|\rt)\leq C_{0}.
\]
Hence, we apply  $\sigma=1$ to Lemma \ref{main lemma 1} to get
\begin{equation}\label{decay_est1}
\|u(t)\|_{L^2}^2 + \|B(t)\|_{L^2}^2 \le C_{0}(1+t)^{-\frac 12}  \quad \text{for all $t>0$.  }
\end{equation}
Then, we use \eqref{decay_est1} and proceed analogously to \eqref{NS_Y_est1} to obtain 
\[
\sup_{0\leq t<\infty}\|u(t)\|_{\my^{\sigma_1}}\leq C_0.
\]
This also implies that for $|\xi|\leq 1$,
\[
\begin{aligned}
|\widehat B(t,\xi)| &\le |\xi|^{-\sigma_2}|\xi|^{\sigma_2} e^{-t|\xi|^2}|\widehat B_0(\xi)| + \int_0^t |\xi| e^{-(t-\tau)|\xi|^2} \|u(\tau)\|_{L^2}\|B(\tau)\|_{L^2}\,d\tau + \int_0^t |\xi|^2 e^{-(t-\tau)|\xi|^2} \|B(\tau)\|_{L^2}^2\,d\tau\\
&\le \|B_0\|_{\my^{\sigma_2}}+ C_{0}\int_0^t (t-\tau)^{-\frac12}(1+\tau)^{-\frac 12}\,d\tau + C_0 \leq C_{0}.
\end{aligned}
\]
Since $\sigma_1\in[0,1]$, we arrive at 
\[
\sup_{0\leq t<\infty}\sup_{|\xi|\leq1}\left(|\xi|^{\sigma_1}|\widehat u(t,\xi)|+|\xi|^{\sigma_1}|\widehat B(t,\xi)|\right)\leq C_0,
\]
and by Lemma \ref{main lemma 1},
\[
\|u(t)\|_{L^2}^2+\|B(t)\|_{L^2}^2\leq C_0(1+t)^{-\frac{3}{2}+\sigma_1}\quad\text{for all $t>0$}.
\]

\noindent $\blacktriangleright$ ({\bf Case 2}: $\sigma_1 \in [-1,0]$, $\sigma_2 \in [-1,0]$): Proceeding as in Case 1, we obtain
\[
\sup_{0\le t <\infty} \sup_{|\xi| \le  1} \lt(|\widehat u(t,\xi)| + |\widehat B(t,\xi)|\rt)<\infty,
\]
and hence
\[
\|u(t)\|_{L^2}^2 + \|B(t)\|_{L^2}^2 \le C(1+t)^{-\frac 32} \quad \text{for all $t>0$.  }
\]
Applying the above inequality and arguing similarly to the derivation of \eqref{NS_Y_est2}, we obtain
\[
\sup_{0\le t <\infty}\|u(t)\|_{\my^{\sigma_{1}}} \leq C_{0}.
\]
Fo $B$, we proceed as follows:
\begin{equation} \label{B bound in Y}
\begin{aligned}
|\xi|^{\sigma_2} |\widehat B(t,\xi)| &\le |\xi|^{\sigma_2} e^{-t|\xi|^2}|\widehat B_0(\xi)| + \int_0^t |\xi|^{1+\sigma_2} e^{-(t-\tau)|\xi|^2} \|u(\tau)\|_{L^2}\|B(\tau)\|_{L^2}\,d\tau\\
&\quad + \int_0^t |\xi|^{2+\sigma_2} e^{-(t-\tau)|\xi|^2} \|B(\tau)\|^{2}_{L^2}\,d\tau\\
&\le \|B_0\|_{\my^{\sigma_2}} + C_{0}\int_0^t (t-\tau)^{-\frac{1+\sigma_2}{2}} (1+\tau)^{-\frac 32} \,d\tau + C_{0}\int_0^t (t-\tau)^{-\frac{2+\sigma_2}{2}} (1+\tau)^{-\frac 32}\,d\tau \leq C_{0}.
\end{aligned}
\end{equation}
Hence we arrive at
\[
\sup_{0\le t<\infty}\sup_{|\xi|\le 1} |\xi|^{\max\{\sigma_1, \sigma_2\}}\lt( |\widehat u(t,\xi)| + |\widehat B(t,\xi)|\rt)\leq C_{0}
\]
from which we obtain the desired temporal decay by Lemma \ref{main lemma 1}. This completes the proof of Theorem \ref{Theorem 3}.

\subsection{Proof of Corollary \ref{Corollary 3}}
We now show that $B \in L^\infty([0,\infty); \mathcal{Y}^{\sigma_2})$, provided that \eqref{index_condition} holds. As demonstrated in \eqref{B bound in Y} for Case 2, we only need to verify the claim for Case 1. From Theorem \ref{Theorem 3} and the fact that $\sigma_1 \geq \sigma_2$ in Case 1, we have
\[
\begin{aligned}
|\xi|^{\sigma_2} |\widehat B(t,\xi)| &\le |\xi|^{\sigma_2} e^{-t|\xi|^2}|\widehat B_0(\xi)| + \int_0^t |\xi|^{1+\sigma_2} e^{-(t-\tau)|\xi|^2} \|u(\tau)\|_{L^2}\|B(\tau)\|_{L^2}\,d\tau \\
&+ \int_0^t |\xi|^{2+\sigma_2} e^{-(t-\tau)|\xi|^2} \|B(\tau)\|^{2}_{L^2} \,d\tau\\
& \le \|B_0\|_{\my^{\sigma_2}} + C_{0}\int_0^t (t-\tau)^{-\frac{1+\sigma_2}{2}} (1+\tau)^{-\frac 32+\sigma_1} \,d\tau  + \int_0^t |\xi|^{2+\sigma_2} e^{-(t-\tau)|\xi|^2} \|B(\tau)\|^{2}_{L^2} \,d\tau\\
&\le  C_{0}+\text{(I)}+\text{(II)}.
\end{aligned}
\]
We first estimate $\text{(I)}$ as follows: 
\eqn \label{time bound 2}
\begin{aligned}
\text{(I)}&\le C_{0}\int_0^t (t-\tau)^{-\frac{1+\sigma_2}{2}}(1+\tau)^{-\frac {1-\sigma_2}{2}}\,d\tau\le C_{0}\int_0^t (t-\tau)^{-\frac{1+\sigma_2}{2}}\tau^{-\frac {1-\sigma_2}{2}}\,d\tau\\
&=C_{0} \int_0^1 (1-\theta)^{-\frac{1+\sigma_2}{2}}\theta^{-\frac {1-\sigma_2}{2}}\,d\theta = C_{0}\mathcal{B}\lt(\frac{1-\sigma_2}{2}, \frac{1+\sigma_2}{2} \rt)\leq C_{0},
\end{aligned}
\een
where we used $\sigma_1 - \frac{\sigma_2}{2}\leq 1$ when $\sigma_2 \neq -1$. If $\sigma_2 =-1$, we use $\sigma_1 - \frac{\sigma_2}{2} = \sigma_1 +\frac12< 1$ to get
\eqn \label{time bound 3}
C_{0}\int_0^t (t-\tau)^{-\frac{1+\sigma_2}{2}}(1+\tau)^{-\frac 32 +\sigma_1}\,d\tau= C_{0}\int_0^t (1+\tau)^{-\frac 32 + \sigma_1}\,d\tau= \frac{C_{0}}{\frac12 - \sigma_1}\lt( 1- (1+t)^{-\frac12 +\sigma_1}\rt)  \leq C_0.
\een
We next bound $\text{(II)}$. For $\sigma_2=0$, one has
\[
\text{(II)} \leq C_0 \int_0^t |\xi|^2 e^{-(t-\tau)|\xi|^2}\,d\tau \le C_0.
\]
For $\sigma_2 <0$, since $\sigma_{1}-\frac{\sigma_{2}}{2}\leq \frac{3}{2}$ holds, we get
\[
\text{(II)}\leq C_{0}\int_0^t (t-\tau)^{-1-\frac{\sigma_2}{2}}(1+\tau)^{\frac{\sigma_{2}}{2}} \,d\tau \leq C_{0}\int_0^t (t-\tau)^{-1-\frac{\sigma_2}{2}}\tau^{\frac{\sigma_{2}}{2}} \,d\tau=C_{0}\mathcal{B}\lt(1+\frac{\sigma_{2}}{2}, -\frac{\sigma_2}{2} \rt)\leq C_{0}
\]
and this implies our desired estimate
\[
\sup_{0\leq t<\infty}\norm{B(t)}{\my^{\sigma_{2}}}\leq C_{0}.
\]

\begin{remark}\upshape
In the proof of Corollary \ref{Corollary 3}, the condition \eqref{index_condition} plays a crucial role. Even without \eqref{index_condition}, we can still show that $B \in L^\infty([0,T); \mathcal{Y}^{\sigma_2})$ for any finite $T > 0$; however, the upper bound for $\|B(t)\|_{\mathcal{Y}^{\sigma_2}}$ depends on $t$:  if $\sigma_1-\frac{\sigma_2}{2}>1$ and $\sigma_2> -1$
\[
\text{(I)}\leq C_0\int^{t}_{0}(t-\tau)^{-\frac{1+\sigma_{2}}{2}}(1+\tau)^{-\frac{3}{2}+\sigma_{1}}\,d\tau \leq C_0 (1+t)^{\sigma_1-\frac{\sigma_2}{2}-1},
\]
and if $\sigma_1\geq \frac{1}{2}$ and $\sigma_2=-1$
\[
\text{(I)}\leq C_0\int^{t}_{0}(1+\tau)^{-\frac{3}{2}+\sigma_{1}}\,d\tau\leq \begin{cases}
C_0(1+t)^{\sigma_1-\frac{1}{2}}\quad \text{if $\sigma_1>\frac{1}{2}$,} \\
C_0\log(1+t)\quad \text{if $\sigma_1=\frac{1}{2}$.}
\end{cases}
\]
\end{remark}

\subsection{Proof of Theorem \ref{Theorem 4}}
As a preliminary step toward proving Theorem \ref{Theorem 4}, we derive the decay rates for the coupled higher-order norms $\|\nabla^{k} u(t)\|_{L^2}^2 + \|\nabla^{k} B(t)\|_{L^2}^2$.

\begin{lemma}\upshape \label{lemma 2}
Under the assumptions in Theorem \ref{Theorem 4}, there exists a unique solution $(u,B)\in L^\infty((0,\infty);H^{\frac{1}{2}}\times H^{\frac{3}{2}})$, and for any $k\in\mathbb{N}\cup \{0\}$ there exists ${T}_k>0$ and $C_k>0$ such that 
\[
\norm{\nabla^{k} u(t)}{L^2}^2+\norm{\nabla^{k} B(t)}{L^2}^2\leq C_{k}(1+t)^{-\frac{3}{2}+\max{\{\sigma_1,\sigma_2\}}-k}\quad\text{for all $t>{T}_k$}.
\]
\end{lemma}

\begin{proof}
Suppose $\|u_0\|_{\dot{H}^{\frac{1}{2}}} + \|B_0\|_{\dot{H}^{\frac{1}{2}}} + \|B_0\|_{\dot{H}^{\frac{3}{2}}}$ is sufficiently small. In this case, the results in \cite{DT22, Tan} guarantee the existence of a global strong solution $(u,B) \in L^\infty([0,\infty); \dot{H}^{\frac{1}{2}} \times (\dot{H}^{\frac{1}{2}} \cap \dot{H}^{\frac{3}{2}}))$. Due to the parabolic smoothing effect, it follows that $(u,B) \in C_b([t_0,\infty); H^k)$ and $(\nabla u, \nabla B) \in L^2(t_0,\infty; H^k)$ for any $t_0 > 0$ and $k \in \mathbb{N}$. Furthermore, by combining Theorem \ref{Theorem 3} with Sobolev interpolation (or alternatively using \cite[Corollary 1.2]{Tan}), we deduce that
\begin{equation}\label{critical_norm_decay}
\|u(t)\|_{\dot{H}^{\frac{1}{2}}}+\|B(t)\|_{\dot{H}^{\frac{1}{2}}}+\|B(t)\|_{\dot{H}^{\frac{3}{2}}}\rightarrow 0\qquad\text{as $t\rightarrow\infty$.}
\end{equation}

Next, we establish the decay rates for $\|\nabla^k u(t)\|_{L^2}^2 + \|\nabla^k B(t)\|_{L^2}^2$. To this end, we first derive energy estimates for each variable, $u$ and $B$, individually. Applying $\nabla^{k}$ to \eqref{Hall MHD a}, taking the inner product with $\nabla^{k}u$, and employing the fractional Leibniz rule \eqref{fractional Leibniz}, we obtain
\begin{equation}\label{u_derivative_energy}
\begin{split}
\frac{1}{2}\frac{d}{dt}\left\|\nabla^{k}u\right\|^{2}_{L^{2}}+\left\|\nabla^{k+1}u\right\|^{2}_{L^{2}} &=-\int\nabla^{k}(u\cdot\nabla u)\cdot\nabla^{k}u+\int\nabla^{k}(B\cdot\nabla B)\cdot\nabla^{k}u\\
&  =\int\nabla^{k}(u\otimes u): \nabla^{k}\nabla u- \int\nabla^{k}(B\otimes B): \nabla^{k}\nabla u\\
&\leq C \|u\|_{L^3}\left\|\nabla^{k+1}u\right\|^{2}_{L^{2}}+C \|B\|_{L^3}\left\|\nabla^{k+1}B\right\|_{L^{2}} \left\|\nabla^{k+1}u\right\|_{L^{2}}.
\end{split}
\end{equation}
Similarly by taking $\nabla^{k}$ to (\ref{Hall MHD b}) and by taking the inner product with $\nabla^{k}B$, and using \eqref{fractional Leibniz}, we obtain 
\begin{equation}\label{B_derivative_energy}
\begin{split}
& \frac{1}{2}\frac{d}{dt}\left\|\nabla^{k}B\right\|^{2}_{L^{2}}+\left\|\nabla^{k+1}B\right\|^{2}_{L^{2}} =  \int\nabla^{k}(u\times B)\cdot\nabla^{k}\nabla\times B-\int\nabla^{k}((\nabla\times B)\times B)\cdot\nabla^{k}\nabla\times B \\
&= \int\nabla^{k}(u\times B)\cdot\nabla^{k}\nabla\times B-\int \left[\nabla^{k}((\nabla\times B)\times B)-\nabla^k(\nabla\times B)\times B\right] \cdot\nabla^{k}\nabla\times B \\
&\leq C \left\|u\right\|_{L^3}\left\|\nabla^{k+1}B\right\|^{2}_{L^{2}}+C \left\|B\right\|_{L^3}\left\|\nabla^{k+1}u\right\|_{L^{2}} \left\|\nabla^{k+1}B\right\|_{L^{2}}+C\left\|\nabla B\right\|_{L^3}\left\|\nabla^{k+1}B\right\|^{2}_{L^{2}}.
\end{split}
\end{equation}
Combining \eqref{u_derivative_energy} and \eqref{B_derivative_energy}, and using $\dot{H}^{\frac{1}{2}}\hookrightarrow L^3$ and \eqref{critical_norm_decay}, we derive
\[
\frac{d}{dt}\left(\|\nabla^k u\|_{L^2}^2+\|\nabla^k B\|_{L^2}^2\right)+\left(\|\nabla^{k+1}u\|_{L^2}^2+\|\nabla^{k+1}B\|_{L^2}^2\right)\leq 0\quad \text{for all $t>\tilde{T}_k$}
\]
for a sufficiently large $\tilde{T}_k>0$. Moreover, from the proof of Theorem \ref{Theorem 3} we have
\[
\begin{aligned}
&\sup_{0\le t<\infty}\sup_{|\xi|\le 1} |\xi|^{\max\{\sigma_1, \sigma_2\}-k}\lt( |\widehat{\nabla^k u}(t,\xi)| + |\widehat{\nabla^k B}(t,\xi)|\rt) \leq \sup_{0\le t<\infty}\sup_{|\xi|\le 1} |\xi|^{\max\{\sigma_1, \sigma_2\}}\lt( |\widehat u(t,\xi)| + |\widehat B(t,\xi)|\rt)\leq C_{0}.
\end{aligned}
\]
Let $N>\frac{3}{2}+k-\max{\{\sigma_1,\sigma_2\}}$ and $T_k:=\max{\{N-1,\tilde{T}_k\}}$. Then, Lemma \ref{main lemma 1} with $f=(\nabla^k u,\nabla^k B)$ completes the proof of Lemma \ref{lemma 2}.  
\end{proof}

\begin{remark}\upshape
In the proof of Lemma \ref{lemma 2}, we set $T_k=\max{\{N-1, \tilde{T}_k\}}$ for any $N>\frac{3}{2}+k-\max{\{\sigma_1,\sigma_2\}}$ and a sufficiently large $\tilde{T}_k>0$. For clarity, we fix 
\begin{equation}\label{N condition}
N>\frac{3}{2}+k-\min{\{\sigma_1,\sigma_2\}},\quad T_k:= \max{\{N-1, \tilde{T}_k\}}
\end{equation}
throughout the proof of Theorem \ref{Theorem 4}.
\end{remark}

\subsubsection*{\bf Proof of Theorem \ref{Theorem 4}}
Now we are ready to prove Theorem \ref{Theorem 4} by estimating $\nabla^k u$ and $\nabla^k B$ separately.

\vspace{1ex}

\noindent
$\blacktriangleright$ \underline{Decay rate of $\nabla^{k} u$}. When $\sigma_1\geq\sigma_2$, the desired decay rates of $\nabla^{k}u$ are straightforward from Lemma \ref{lemma 2}. So, we only consider the case $\sigma_1<\sigma_2\leq0$. From \eqref{u_derivative_energy}, we obtain
\[
\frac{1}{2}\frac{d}{dt}\left\|\nabla^{k}u\right\|^{2}_{L^{2}}+\left\|\nabla^{k+1}u\right\|^{2}_{L^{2}} \leq C \|u\|_{L^3}\left\|\nabla^{k+1}u\right\|^{2}_{L^{2}}+\frac{1}{4} \left\|\nabla^{k+1}u\right\|^{2}_{L^{2}} +C \|B\|^{2}_{L^3}\left\|\nabla^{k+1}B\right\|^{2}_{L^{2}}.
\]
Hence, after using \eqref{critical_norm_decay} we deduce from Theorem \ref{Theorem 3} and Lemma \ref{lemma 2} that for all $t> T_k$ 
\[
\frac{d}{dt}\left\|\nabla^{k}u\right\|^{2}_{L^{2}}+\left\|\nabla^{k+1}u\right\|^{2}_{L^{2}} \leq C \|B\|^{2}_{L^3}\left\|\nabla^{k+1}B\right\|^{2}_{L^{2}}\leq C \|B\|_{L^2}\|\nabla B\|_{L^2}\left\|\nabla^{k+1}B\right\|^{2}_{L^{2}}\leq C_{k} (1+t)^{-\frac{9}{2} +2\sigma_2 -k}.
\]
By modifying Lemma \ref{main lemma 1}, we can derive 
\[
\begin{aligned}
\frac{d}{dt}\left[(1+t)^N\norm{\nabla^k u(t)}{L^2}^2\right] &\leq N(1+t)^{N-1}\int_{\{|\xi|^2\leq\frac{N}{1+t}\}}|\xi|^{2k}|\widehat{u}(t,\xi)|^2+C_{k} (1+t)^{N-\frac{9}{2} +2\sigma_2 -k}\\
&\leq C_k(1+t)^{N-\frac{5}{2}+\sigma_1-k}+C_{k} (1+t)^{N-\frac{9}{2} +2\sigma_2 -k}\leq C_{k} (1+t)^{N-\frac{5}{2} +\sigma_1 -k}
\end{aligned}
\]
for $N$ fixed in \eqref{N condition}. By integrating this in time, we deduce that
\[
\norm{\nabla^k u(t)}{L^2}^2\leq C_k(1+t)^{-\frac{3}{2}+\sigma_1-k}\quad\text{for all $t>T_k$}.
\]

\noindent
$\blacktriangleright$ \underline{Decay rate of $\nabla^{k} B$}. When $\sigma_1\leq \sigma_2$, the desired decay rates of $\nabla^k B$ are obvious from Lemma \ref{lemma 2}, and so is $B\in L^{\infty}([0,\infty);\my^{\sigma_2})$ from \eqref{B bound in Y}. So, we only consider $\sigma_1>\sigma_2$. We first check that
\begin{equation}\label{B_decay}
B\in L^{\infty}([0,\infty);\my^{\sigma_2})\quad\text{and}\quad \|B(t)\|_{L^2}^2\leq C_0(1+t)^{-\frac{3}{2}+\sigma_2}.
\end{equation}
When $\sigma_1>\sigma_2$, Theorem \ref{Theorem 3} gives
\[
\norm{u(t)}{L^2}^2+\norm{B(t)}{L^2}^2\leq C_0(1+t)^{-\frac{3}{2}+\sigma_1}.
\]
By taking $L^2$ inner product of \eqref{Hall MHD b} with $B$, we obtain 
\begin{align*}
\frac{1}{2}\frac{d}{dt}\norm{B}{L^2}^2+\norm{\nabla B}{L^2}^2=\int (B\cdot\nabla u)\cdot B\leq C\norm{u}{L^3}\norm{\nabla B}{L^2}^2\leq \frac{1}{2}\|\nabla B\|_{L^2}^2\quad\text{for all $t>T_0$}
\end{align*}
for some $T_0>0$ from \eqref{critical_norm_decay}.
Since Theorem \ref{Theorem 3} says $$\sup_{0\leq t<\infty}\sup_{|\xi|\leq 1}|\widehat B(t,\xi)|\leq C_0,$$ Lemma \ref{main lemma 1} yields
\[
\norm{B(t)}{L^2}^2\leq C_0(1+t)^{-\frac{3}{2}}.
\]
If $\sigma_2\neq -1$, then $\sigma_1-\sigma_2<2$. So, we obtain
\begin{align*}
|\xi|^{\sigma_2} |\widehat B(t,\xi)| &\le \norm{B_0}{\my^{\sigma_2}} + \int_0^t |\xi|^{1+\sigma_2} e^{-(t-\tau)|\xi|^2} \|u(\tau)\|_{L^2}\|B(\tau)\|_{L^2}\,d\tau+ \int_0^t |\xi|^{2+\sigma_2} e^{-(t-\tau)|\xi|^2} \|B(\tau)\|^{2}_{L^2}\,d\tau\\
&\le \|B_0\|_{\my^{\sigma_2}} + C_{0}\int_0^t (t-\tau)^{-\frac{1+\sigma_2}{2}} (1+\tau)^{-\frac 32+\frac{\sigma_1}{2}} \,d\tau + C_{0}\int_0^t (t-\tau)^{-\frac{2+\sigma_2}{2}} (1+\tau)^{-\frac 32}\,d\tau \leq C_{0},
\end{align*}
where we use 
\[
\int_0^t (t-\tau)^{-\frac{1+\sigma_2}{2}} (1+\tau)^{-\frac 32+\frac{\sigma_1}{2}} \,d\tau\leq \int_0^t (t-\tau)^{-\frac{1+\sigma_2}{2}} \tau^{-\frac 12+\frac{\sigma_2}{2}} \,d\tau=\mathcal{B}\left(\frac{1-\sigma_2}{2},\frac{1+\sigma_2}{2}\right).
\]
Hence, by Lemma \ref{main lemma 1}
\[
\norm{B(t)}{L^2}^2\leq C_0(1+t)^{-\frac{3}{2}+\sigma_2}.
\]
If $\sigma_2=-1$, when $|\xi|\leq 1$ we have
\begin{align*}
|\xi|^{-\frac{1}{2}}|\widehat B(t,\xi)| &\leq \norm{B_0}{\my^{-1}} + \int_0^t |\xi|^{\frac{1}{2}} e^{-(t-\tau)|\xi|^2} \|u(\tau)\|_{L^2}\|B(\tau)\|_{L^2}\,d\tau+ \int_0^t |\xi|^{\frac{3}{2}} e^{-(t-\tau)|\xi|^2} \|B(\tau)\|^{2}_{L^2}\,d\tau\\
&\le \|B_0\|_{\my^{-1}} + C_{0}\int_0^t (t-\tau)^{-\frac{1}{4}} (1+\tau)^{-\frac 32+\frac{\sigma_1}{2}}\,d\tau + C_{0}\int_0^t (t-\tau)^{-\frac{3}{4}} (1+\tau)^{-\frac 32}\,d\tau \leq C_{0}
\end{align*}
which implies 
\[
\norm{B(t)}{L^2}^2\leq C_0(1+t)^{-2}.
\]
Thus, we bound $$\sup_{0\leq t<\infty}\norm{B(t)}{\my^{-1}}\leq C_0$$ and obtain the desired decay rate of $B$ by Lemma \ref{main lemma 1}.
 
 \vspace{1ex}

Now, we compute the decay rates of $\nabla^k B$ for $k\in\mathbb{N}$. From \eqref{B_derivative_energy}, we have
\[
\begin{split}
& \frac{1}{2}\frac{d}{dt}\left\|\nabla^{k}B\right\|^{2}_{L^{2}}+\left\|\nabla^{k+1}B\right\|^{2}_{L^{2}} \\
&\leq C \left\|u\right\|_{L^3}\left\|\nabla^{k+1}B\right\|^{2}_{L^{2}}+C \left\|B\right\|^{2}_{L^3}\left\|\nabla^{k+1}u\right\|^{2}_{L^{2}}+\frac{1}{4} \left\|\nabla^{k+1}B\right\|^{2}_{L^{2}}+C\left\|\nabla B\right\|_{L^3}\left\|\nabla^{k+1}B\right\|^{2}_{L^{2}}.
\end{split}
\]
By applying \eqref{critical_norm_decay}, we deduce from \eqref{B_decay} and Lemma \ref{lemma 2} that for all $t > T_k$
\[
\frac{d}{dt}\left\|\nabla^{k}B\right\|^{2}_{L^{2}}+\left\|\nabla^{k+1}B\right\|^{2}_{L^{2}}\leq C\norm{B}{L^2}\norm{\nabla B}{L^2}\norm{\nabla^{k+1}u}{L^2}^{2} \leq C_{k} (1+t)^{-\frac{9}{2} +\frac{3\sigma_1+\sigma_2}{2}-k}.
\]
As we bound $\nabla^k u$ above, for $N$ in \eqref{N condition} and for all $t>T_k$, we attain
\[
\frac{d}{dt}\left[(1+t)^N\|\nabla^{k}B(t)\|_{L^{2}}^2  \right] \leq C_k(1+t)^{N-\frac{5}{2} +\sigma_{2}-k}+C_{k} (1+t)^{N-\frac{9}{2} +\frac{3\sigma_1+\sigma_2}{2} -k}\leq C_k(1+t)^{N-\frac{5}{2} +\sigma_{2}-k}
\]
since $3\sigma_1-\sigma_2\leq 4$. Hence, we derive 
\[
\norm{\nabla^{k}B(t)}{L^2}^2\leq C_k(1+t)^{-\frac{3}{2}+\sigma_2-k}\quad\text{for all $t>T_k$}.
\]
This completes the proof of Theorem \ref{Theorem 4}.

\subsection{Proof of Theorem \ref{Theorem 5}}
From \cite{Liu}, we deduce that if $\|u_0\|_{\mx^{-1}} + \|B_0\|_{\mx^{-1}}+ \|B_0\|_{\mx^0}\leq \epsilon$ is sufficiently small, $u \in C([0,\infty); \mx^{-1}) \cap L^1((0,\infty); \mx^1)$, $B \in C([0,\infty); \mx^{-1}\cap \mx^0) \cap L^1((0,\infty); \mx^1\cap \mx^2)$, and 
\eqn \label{Hall MHD bound in X}
\left\|u(t)\right\|_{\mx^{-1}}+\left\|B(t)\right\|_{\mx^{-1}}+\left\|B(t)\right\|_{\mx^{0}} +(1-C\epsilon)\int^{t}_{0}\left(\left\|u(\tau)\right\|_{\mx^{1}}+\left\|B(\tau)\right\|_{\mx^{1}}+\left\|B(\tau)\right\|_{\mx^{2}}\right)d\tau \leq C\epsilon
\een
for all $t>0$. So, we only need to show $(u,B) \in L^{\infty}([0,\infty); \my^{\sigma})$ and derive the decay rate in $\mx^{-1}$.

\subsubsection{\bf Bounds in $\my^\sigma$} \label{Bound Y sigma dd}
We divide the estimates for $(u, B)$ into three cases, following the approach in the proof of Theorem \ref{Theorem 3}. To this end, we repeatedly employ \eqref{Hall MHD bound in X} to bound $(u, B)$ in $\mathcal{Y}^\sigma$. However, since $u$ can be bounded analogously to the proof of Theorem \ref{Theorem 2}, we focus primarily on the details of estimating $B$.

\vspace{1ex}
\noindent $\blacktriangleright$ ({\bf Case 1}: $\sigma=1$) We first have
\[
|\xi||\widehat u(t,\xi)| \leq \|u_0\|_{\my^1}+ C\epsilon \sup_{0 \le \tau \le t}\lt( \|u(\tau)\|_{\my^1} + \|B(\tau)\|_{\my^1}\rt).
\]
We now estimate $B$. In this case, we use
\[
B(t)=e^{t\Delta}B_{0}+\int^{t}_{0}e^{(t-\tau)\Delta}\left(\nabla \times (u\times B) -\nabla \times ((\nabla \times B)\times B)\right)(\tau)\,d\tau
\]
instead of (\ref{HMHD Integral b}). Then, using (\ref{est_conv 1}) with $\alpha=1$, we have 
\[
\begin{aligned}
|\xi||\widehat B(t,\xi)| &\le |\xi|e^{-t|\xi|^2}|\widehat B_0(\xi)| + \int_0^t  |\xi|^2 e^{-(t-\tau)|\xi|^2}\intr |\widehat  u(\tau,\xi-\eta)| |\widehat B(\tau,\eta)| \,d\eta d\tau\\
&\quad + \int_0^t |\xi|^2 e^{-(t-\tau)|\xi|^2}\intr  |\widehat{\nabla \times B}(\tau,\xi-\eta)| |\widehat B(\tau,\eta)|\,d\eta d\tau\\
&\le \|B_0\|_{\my^1} + \sup_{0\le \tau\le t} \left(\|u(\tau)\|_{\my^1}\|B(\tau)\|_{\mx^{-1}} + \|B(\tau)\|_{\my^1}\|u(\tau)\|_{\mx^{-1}}\right) \int_0^t |\xi|^2 e^{-(t-\tau)|\xi|^2}\,d\tau\\
&\quad + \sup_{0\le \tau \le t} \left(\|B(\tau)\|_{\my^1}\|B(\tau)\|_{\mx^0}\right) \int_0^t |\xi|^2 e^{-(t-\tau)|\xi|^2}\,d\tau\\
& \leq \|B_0\|_{\my^1}+ C\epsilon \sup_{0 \le \tau \le t}\lt( \|u(\tau)\|_{\my^1} + \|B(\tau)\|_{\my^1}\rt).
\end{aligned}
\]
From these two bounds, we have
\[
\sup_{0 \le \tau \le t}\lt( \|u(\tau)\|_{\my^1} + \|B(\tau)\|_{\my^1}\rt)\leq \|u_0\|_{\my^1} + \|B_0\|_{\my^1}+C\epsilon \sup_{0 \le \tau \le t}\lt( \|u(\tau)\|_{\my^1} + \|B(\tau)\|_{\my^1}\rt).
\]

\noindent $\blacktriangleright$ ({\bf Case 2}: $\sigma \in (-1,1)$) In this case, we bound $u$ as 
\[
\|u(t)\|_{\my^\sigma}\leq \|u_0\|_{\my^\sigma}+ C\epsilon \sup_{0 \le \tau \le t}\lt( \|u(\tau)\|_{\my^\sigma} + \|B(\tau)\|_{\my^\sigma}\rt).
\]
For $B$, we have
\begin{align*}
|\xi|^\sigma |\widehat B(t,\xi)| &\le |\xi|^\sigma e^{-t|\xi|^2}|\widehat B_0(\xi)| + \int_0^t  |\xi|^{1+\sigma} e^{-(t-\tau)|\xi|^2}\int |\widehat  u(\tau,\xi-\eta)| |\widehat B(\tau,\eta)| \,d\eta d\tau\\
&\quad + \int_0^t |\xi|^{2+\sigma} e^{-(t-\tau)|\xi|^2}\int  |\widehat B(\tau,\xi-\eta)| |\widehat B(\tau,\eta)|\,d\eta d\tau\le \|B_0\|_{\my^\sigma}+\text{(I)} +\text{(II)}.
\end{align*}
Using \eqref{est_conv 1} with $\alpha=\sigma$ and \eqref{intp_x 1}, we first bound (I) as
\begin{align*}
\text{(I)} &\leq \int_0^t  |\xi|^{1+\sigma} e^{-(t-\tau)|\xi|^2} \left(\|u(\tau)\|_{\my^\sigma}\|B(\tau)\|_{\mx^{-\sigma}}+\|B(\tau)\|_{\my^{\sigma}}\|u(\tau)\|_{\mx^{-\sigma}}\right)\,d\tau \\
&\leq \sup_{0\le \tau \le t}\|u(\tau)\|_{\my^\sigma}  \Big(\int_0^t  |\xi|^2 e^{-(t-\tau)\frac{2}{1+\sigma}|\xi|^2} \,d\tau\Big)^{\frac{1+\sigma}{2}} \Big(\int_0^t \|B(\tau)\|_{\mx^{-\sigma}}^{\frac{2}{1-\sigma}}\,d\tau \Big)^{\frac{1-\sigma}{2}}\\
&\quad +  \sup_{0\le \tau \le t}\|B(\tau)\|_{\my^{\sigma}}  \Big(\int_0^t  |\xi|^2 e^{-(t-\tau)\frac{2}{1+\sigma}|\xi|^2} \,d\tau\Big)^{\frac{1+\sigma}{2}}\Big(\int_0^t \|u(\tau)\|_{\mx^{-\sigma}}^{\frac{2}{1-\sigma}}\,d\tau \Big)^{\frac{1-\sigma}{2}}\\
&\leq C\sup_{0\le \tau \le t}\|u(\tau)\|_{\my^\sigma} \Big(\sup_{0\le \tau \le t}\|B(\tau)\|_{\mx^{-1}} \Big)^{\frac{1+\sigma}{2}} \Big(\int_0^t \|B(\tau)\|_{\mx^1}\,d\tau \Big)^{\frac{1-\sigma}{2}}\\
& \quad+ C \sup_{0\le \tau \le t}\|B(\tau)\|_{\my^\sigma} \Big(\sup_{0\le \tau \le t}\|u(\tau)\|_{\mx^{-1}} \Big)^{\frac{1+\sigma}{2}} \Big(\int_0^t \|u(\tau)\|_{\mx^1}\,d\tau \Big)^{\frac{1-\sigma}{2}}\\
&\leq C\epsilon \sup_{0 \le \tau \le t}\lt( \|u(\tau)\|_{\my^\sigma} + \|B(\tau)\|_{\my^\sigma}\rt).
\end{align*}

When $\sigma\in[0,1)$, using
\eqn \label{sigma inequality}
|\xi|^\sigma \le (|\xi -\eta| + |\eta|)^\sigma \le  C(|\xi-\eta|^\sigma + |\eta|^\sigma),
\een
we bound
\begin{align*}
\text{(II)} &\le C\int_0^t |\xi|^2 e^{-(t-\tau)|\xi|^2}  \int |\xi-\eta|^\sigma  |\widehat B(\tau,\xi-\eta)| |\widehat B(\tau,\eta)|\,d\eta d\tau\\
&\quad + C\int_0^t |\xi|^2 e^{-(t-\tau)|\xi|^2}  \int |\widehat B(\tau,\xi-\eta)| |\eta|^\sigma |\widehat B(\tau,\eta)|\,d\eta d\tau\\
&\le C \sup_{0 \le \tau \le t}\|B(\tau)\|_{\my^\sigma}  \sup_{0 \le \tau \le t}\|B(\tau)\|_{\mx^0}\int_0^t |\xi|^2 e^{-(t-\tau)|\xi|^2} \,d\tau \leq C\epsilon \sup_{0 \le \tau \le t}\|B(\tau)\|_{\my^\sigma}.
\end{align*}

When $\sigma \in (-1,0)$, we use (\ref{intp_x 2})  to obtain 
\begin{align*}
\text{(II)} &\le \int_0^t |\xi|^{2+\sigma} e^{-(t-\tau)|\xi|^2}\| B(\tau)\|_{\my^\sigma} \|B(\tau)\|_{\mx^{-\sigma}}\,d\tau\\
&\le \sup_{0\le \tau \le t}\|B(\tau)\|_{\my^\sigma} \Big(\int_0^t |\xi|^2 e^{-(t-\tau)\frac{2}{2+\sigma}|\xi|^2 }\,d\tau\Big)^{\frac{2+\sigma}{2}} \Big(\int_0^t \|B(\tau)\|_{\mx^{-\sigma}}^{\frac{2}{-\sigma}}\,d\tau  \Big)^{\frac{-\sigma}{2}}\\
&\le C\sup_{0\le \tau \le t}\|B(\tau)\|_{\my^\sigma} \Big(\sup_{0\le \tau \le t}\|B(\tau)\|_{\mx^0} \Big)^{\frac{2+\sigma}{2}} \Big(\int_0^t \|B(\tau)\|_{\mx^2}\,d\tau  \Big)^{\frac{-\sigma}{2}} \leq C\epsilon \sup_{0 \le \tau \le t}\|B(\tau)\|_{\my^\sigma}.
 \end{align*}

Combining the above estimates, we arrive at
\[
\sup_{0 \le \tau \le t}\lt( \|u(\tau)\|_{\my^\sigma} + \|B(\tau)\|_{\my^\sigma}\rt)\leq \|u_0\|_{\my^\sigma} + \|B_0\|_{\my^\sigma} +C\epsilon \sup_{0 \le \tau \le t}\lt( \|u(\tau)\|_{\my^\sigma} + \|B(\tau)\|_{\my^\sigma}\rt).
\]

\noindent $\blacktriangleright$ ({\bf Case 3}: $\sigma=-1$) In this case, we simply use \eqref{est_conv 1} with $\alpha=1$ to get
\[
|\xi|^{-1} |\widehat u(t,\xi)| \le \|u_0\|_{\my^{-1}} + \sup_{0\le \tau \le t}\|u(\tau)\|_{\my^{-1}} \int_0^t \|u(\tau)\|_{\mx^1} \,d\tau + \sup_{0\le \tau \le t}\|B(\tau)\|_{\my^{-1}}\int_0^t \|B(\tau)\|_{\mx^1} \,d\tau,
\]
and
\begin{align*}
|\xi|^{-1} |\widehat B(t,\xi)| &\le \|B_0\|_{\my^{-1}} +  \sup_{0\le \tau \le t}\|u(\tau)\|_{\my^{-1}}\int_0^t \|B(\tau)\|_{\mx^1}\,d\tau +  \sup_{0\le \tau \le t}\|B(\tau)\|_{\my^{-1}} \int_0^t \|u(\tau)\|_{\mx^1} \,d\tau\\
&\quad + \int_0^t |\xi| e^{-(t-\tau)|\xi|^2}\| B(\tau)\|_{\my^{-1}} \|B(\tau)\|_{\mx^1} d\tau\\
&\le \|B_0\|_{\my^{-1}} +  \sup_{0\le \tau \le t}\|u(\tau)\|_{\my^{-1}}\int_0^t \|B(\tau)\|_{\mx^1}\,d\tau +   \sup_{0\le \tau \le t}\|B(\tau)\|_{\my^{-1}}  \int_0^t \|u(\tau)\|_{\mx^1} \,d\tau\\
&\quad +\sup_{0\le \tau \le t}\|B(\tau)\|_{\my^{-1}}  \Big(\int_0^t |\xi|^2 e^{-2(t-\tau)|\xi|^2}\,d\tau\Big)^{\frac12} \Big(\int_0^t \|B(\tau)\|_{\mx^1}^2 d\tau\Big)^{\frac12}\\
&\le  \|B_0\|_{\my^{-1}} +  \sup_{0\le \tau \le t}\|u(\tau)\|_{\my^{-1}}\int_0^t \|B(\tau)\|_{\mx^1}\,d\tau +  \sup_{0\le \tau \le t}\|B(\tau)\|_{\my^{-1}}\int_0^t \|u(\tau)\|_{\mx^1} \,d\tau\\
&\quad +C\sup_{0\le \tau \le t}\|B(\tau)\|_{\my^{-1}} \lt(\sup_{0\le \tau \le t}\|B(\tau)\|_{\mx^0}\rt)^{\frac{1}{2}}  \lt(\int_0^t \|B(\tau)\|_{\mx^2} \,d\tau\rt)^{\frac12}.
\end{align*}
Therefore, we have
\[
\sup_{0 \le \tau \le t}\lt( \|u(\tau)\|_{\my^{-1}} + \|B(\tau)\|_{\my^{-1}}\rt)\leq \|u_0\|_{\my^{-1}} + \|B_0\|_{\my^{-1}} +C\epsilon \sup_{0 \le \tau \le t}\lt( \|u(\tau)\|_{\my^{-1}} + \|B(\tau)\|_{\my^{-1}}\rt).
\]

\noindent $\blacktriangleright$ By collecting the bounds together, we arrive at 
\[
\sup_{0 \le t <\infty}\lt( \|u(t)\|_{\my^{\sigma}} + \|B(t)\|_{\my^{\sigma}}\rt)\leq \|u_0\|_{\my^{\sigma}} + \|B_0\|_{\my^{\sigma}} +C\epsilon \sup_{0 \le t <\infty}\lt( \|u(t)\|_{\my^{\sigma}} + \|B(t)\|_{\my^{\sigma}}\rt)
\]
for all $\sigma\in [-1,1]$. By restricting the size of $\epsilon$ as $2C\epsilon< 1$, we finally obtain 
\[
\sup_{0 \le t<\infty}\lt( \|u(t)\|_{\my^{\sigma}} + \|B(t)\|_{\my^{\sigma}}\rt)\leq 2 \left(\|u_0\|_{\my^{\sigma}} + \|B_0\|_{\my^{\sigma}} \right).
\]

\subsubsection{\bf Decay rates}
We now investigate the temporal decay rates of $(u,B)$ in $\mathcal{X}^{-1}$ by applying Lemma \ref{main lemma 2}. Since the uniform bound for $(u,B)$ in $\mathcal{Y}^\sigma$ implies \eqref{X inequality 2}, it suffices to show that $(u,B)$ satisfies \eqref{X inequality} with $k=0$. Indeed, by employing \eqref{est_conv 1} with $\alpha=1$, we obtain
\[
\begin{aligned}
\frac{d}{dt}\|u(t)\|_{\mx^{-1}} + \|u(t)\|_{\mx^1} &\le \int\int |\widehat u (t, \xi-\eta)| |\widehat u(t,\eta)| \,d\eta d\xi + \int\int |\widehat B (t, \xi-\eta)| |\widehat B(t,\eta)| \,d\eta d\xi\\
&\le \|u(t)\|_{\mx^{-1}}\|u(t)\|_{\mx^1} + \|B(t)\|_{\mx^{-1}}\|B(t)\|_{\mx^1} \leq C\epsilon \left(\|u(t)\|_{\mx^1} +\|B(t)\|_{\mx^1}\right)
\end{aligned}
\]
and
\[
\begin{aligned}
\frac{d}{dt}\|B(t)\|_{\mx^{-1}} +  \|B(t)\|_{\mx^1} &\le 2\int\int |\widehat u (t, \xi-\eta)| |\widehat B(t,\eta)| \,d\eta d\xi+ \int\int |\widehat{\nabla  \times B} (t, \xi-\eta)| |\hat B(t,\eta)| \,d\eta d\xi\\
&\le \|u(t)\|_{\mx^{-1}}\|B(t)\|_{\mx^1} + \|B(t)\|_{\mx^{-1}}\|u(t)\|_{\mx^1} +\|B(t)\|_{\mx^0}\|B(t)\|_{\mx^1}\\
&\leq C\epsilon \left(\|u(t)\|_{\mx^1} +\|B(t)\|_{\mx^1}\right).
\end{aligned}
\]
By combining these two estimates with $\epsilon$ restricted to satisfy $\theta=1-C\epsilon>0$, we have
\[
\frac{d}{dt}\lt( \|u(t)\|_{\mx^{-1}} + \|B(t)\|_{\mx^{-1}} \rt) + \theta \lt( \|u(t)\|_{\mx^1} + \|B(t)\|_{\mx^1} \rt) \le 0.
\]
Therefore, the desired decay rate in $\mathcal{X}^{-1}$ is established by Lemma \ref{main lemma 2}.

\subsubsection{\bf Decay rates in $\mx^{k-1}$}
Since we already have the uniform bound of $(u, B)$ in $\my^\sigma$, we only need to show that $(u,B)$ satisfies (\ref{X inequality}). Using (\ref{sigma inequality}) and (\ref{intp_x 3}),
\[
\begin{aligned}
\frac{d}{dt}\|u(t)\|_{\mx^{k-1}} + \|u(t)\|_{\mx^{k+1}} &\le \int\int |\xi|^{k}|\widehat u (t, \xi-\eta)| |\widehat u(t,\eta)| \,d\eta d\xi + \int\int |\xi|^{k}|\widehat B (t, \xi-\eta)| |\widehat B(t,\eta)| \,d\eta d\xi\\
&\le C\|u(t)\|_{\mx^{0}}\|u(t)\|_{\mx^k} + C\|B(t)\|_{\mx^{0}}\|B(t)\|_{\mx^k} \\
&\le C\|u(t)\|_{\mx^{-1}}\|u(t)\|_{\mx^{k+1}} + C\|B(t)\|_{\mx^{-1}}\|B(t)\|_{\mx^{k+1}} \\
&\leq C\epsilon \left(\|u(t)\|_{\mx^{k+1}} +\|B(t)\|_{\mx^{k+1}}\right)
\end{aligned}
\]
and
\[
\begin{aligned}
\frac{d}{dt}\|B(t)\|_{\mx^{k-1}} +  \|B(t)\|_{\mx^{k+1}} &\le 2\int\int |\xi|^{k}|\widehat u (t, \xi-\eta)| |\widehat B(t,\eta)| \,d\eta d\xi+ \int\int |\xi|^{k+1}|\widehat{B} (t, \xi-\eta)| |\hat B(t,\eta)| \,d\eta d\xi\\
&\le C\|u(t)\|_{\mx^{0}}\|B(t)\|_{\mx^k} + C\|B(t)\|_{\mx^{0}}\|u(t)\|_{\mx^k}+C\|B(t)\|_{\mx^{0}}\|B(t)\|_{\mx^{k+1}} \\
&\leq C\|u(t)\|_{\mx^{-1}}\|u(t)\|_{\mx^{k+1}} + C\|B(t)\|_{\mx^{-1}}\|B(t)\|_{\mx^{k+1}}+C\|B(t)\|_{\mx^{0}}\|B(t)\|_{\mx^{k+1}}\\
&\leq C\epsilon \left(\|u(t)\|_{\mx^{k+1}} +\|B(t)\|_{\mx^{k+1}}\right).
\end{aligned}
\]
By combining these two estimates with $\epsilon$ restricted to satisfy $\theta=1-C\epsilon>0$, we have
\[
\frac{d}{dt}\lt( \|u(t)\|_{\mx^{k-1}} + \|B(t)\|_{\mx^{k-1}} \rt) + \theta \lt( \|u(t)\|_{\mx^{k+1}} + \|B(t)\|_{\mx^{k+1}} \rt) \le 0
\]
and again, the decay rate follows from Lemma \ref{main lemma 2}. This completes the proof of Theorem \ref{Theorem 5}.

\appendix

\section{$L^{1}$ bound of $u$} \label{A 1}
The decay rate in \cite{Schonbek}  implies that  $L^{1}$ is an invariant space for all $t>0$. To show this, we first note that 
\[
\frac{d}{dt} \left\|u\right\|_{L^{1}}\leq \left\|\mathbb{P}(u\cdot \nabla u)\right\|_{L^{1}},
\]
where $\mathbb{P}$ is  the  Leray projection operator defined in (\ref{Leray projection}). Since $u\cdot \nabla u\in \mathcal{H}$ (the Hardy space) with $\norm{u\cdot\nabla u}{\mathcal{H}}\leq C \norm{u}{L^2}\norm{\nabla u}{L^2}$ \cite{Coifman} and $\mathbb{P}: \mathcal{H} \rightarrow L^{1}$ \cite{Stein}, we obtain 
\[
\frac{d}{dt} \left\|u\right\|_{L^{1}}\leq C\left\|u\cdot \nabla u\right\|_{\mathcal{H}} \leq C \left\|u\right\|_{L^{2}}\left\|\nabla u\right\|_{L^{2}}.
\]
Integrating this in time, 
\[
\begin{split}
\left\|u(t)\right\|_{L^{1}}&\leq \left\|u_{0}\right\|_{L^{1}}+C\int^{t}_{0}(1+\tau)^{-\frac{3}{4}} \left\|\nabla u(\tau)\right\|_{L^{2}}d\tau\\
&\leq \left\|u_{0}\right\|_{L^{1}}+C\left[\int^{t}_{0}\left\|\nabla u(\tau)\right\|^{2}_{L^{2}}d\tau\right]^{\frac{1}{2}} \left[\int^{t}_{0}(1+\tau)^{-\frac{3}{2}} d\tau\right]^{\frac{1}{2}}\leq \left\|u_{0}\right\|_{L^{1}}+C \left\|u_{0}\right\|_{L^{2}} \quad \text{for all $t>0$. }
\end{split}
\]

\section{Alternative proof of Theorem \ref{Theorem 1}} \label{A 2}

To prove Theorem \ref{Theorem 1}, we take two steps. We first estimate $\left\|e^{t\Delta}u_{0}\right\|_{L^{2}}$ using $u_{0}\in L^{2}\cap \my^{\sigma}$ from which we arrive at the decay rate of $u$ in Theorem \ref{Theorem 1} by \cite{Wiegner}. We then bound $u\in \my^{\sigma}$ using the decay rate of $u$.

\vspace{1ex}

\noindent
$\blacktriangleright$ We recall the argument in \cite{Wiegner}: if $\left\|e^{t\Delta}u_{0}\right\|^{2}_{L^{2}} \leq C(1+t)^{-\alpha_{0}}$, then  a weak solution of (\ref{NSE}) decays in time:
\[
\|u(t)\|^{2}_{L^{2}}\leq C_{0}(1+t)^{-\min \left\{\alpha_{0}, \frac{5}{2}\right\}}.
\]

We now bound $\left\|e^{t\Delta}u_{0}\right\|_{L^{2}}$ with $u_{0}\in L^{2}\cap \my^{\sigma}$: 
\[
\norm{e^{t\Delta}u_{0}}{L^{2}}^{2} =\int e^{-2t|\xi|^{2}}|\widehat{u}_{0}(\xi)|^{2}\,d\xi=\int_{|\xi|\leq 1} e^{-2t|\xi|^{2}}|\widehat{u}_{0}(\xi)|^{2}\,d\xi+\int_{|\xi|\ge 1} e^{-2t|\xi|^{2}}|\widehat{u}_{0}(\xi)|^{2}\,d\xi=\text{(I)+(II)}.
\]
We first deal with $\text{(I)}$: 
\[
\begin{split}
\text{(I)}&=(2t)^{-\frac{3}{2}+\sigma}\int_{|\xi|\leq 1}(2t|\xi|^{2})^{-\sigma}e^{-2t|\xi|^{2}}\left(|\xi|^{\sigma}|\widehat{u}_{0}(\xi)|\right)^{2}(2t)^{\frac{3}{2}}\,d\xi \\
&\leq Ct^{-\frac{3}{2}+\sigma}\norm{u_{0}}{\mathcal{Y}^{\sigma}}^{2} \int_{|\eta|\leq\sqrt{2t}}|\eta|^{-2\sigma} e^{-|\eta|^{2}}d\eta \leq Ct^{-\frac{3}{2}+\sigma}\norm{u_{0}}{\mathcal{Y}^{\sigma}}^{2},
\end{split}
\]
where we use the change of variables $\eta =\sqrt{2t} \xi$ and 
\[
\int |\eta|^{-2\sigma}e^{-|\eta|^{2}}\,d\eta=C\int^{\infty}_{0}r^{2-2\sigma} e^{-r^{2}}dr\leq C\int^{1}_{0}r^{2-2\sigma}dr +C\int^{\infty}_{1}e^{-r^{2}}dr\leq C
\]
when $\sigma<\frac{3}{2}$. Using \eqref{heat_inf_est} and the condition $\sigma<\frac{3}{2}$, we now bound $\text{(II)}$: 
\[
\text{(II)} =(2t)^{-\frac{3}{2}+\sigma}\int_{|\xi|\geq 1}(2t|\xi|^{2})^{\frac{3}{2}-\sigma}e^{-2t|\xi|^{2}}|\xi|^{2\sigma-3}|\widehat{u}_{0}(\xi)|^{2}\,d\xi  \leq Ct^{-\frac{3}{2}+\sigma} \int |\widehat{u}_{0}(\xi)|^{2}\,d\xi \leq  Ct^{-\frac{3}{2}+\sigma} \norm{u_{0}}{L^{2}}^{2}.   
\]
Therefore, these two bounds lead to  
\[
\norm{e^{t\Delta}u_{0}}{L^{2}}^{2} \leq Ct^{-\frac{3}{2}+\sigma} \left(\norm{u_{0}}{\mathcal{Y}^{\sigma}}^{2}+\norm{u_{0}}{L^{2}}^{2} \right).
\]
Since $\alpha_{0}=\frac{3}{2}-\sigma \leq \frac{5}{2}$ when $\sigma\ge -1$, we arrive at the desired decay rate of $u$ in Theorem \ref{Theorem 1}.

\vspace{1ex}

\noindent
$\blacktriangleright$ We here show that $u\in \my^{\sigma}$ using  (\ref{Theorem 1 Decay}). When $\sigma=1$, we  first have
\[
\begin{aligned}
|\xi| |\widehat u(t,\xi)| &\leq |\xi|e^{-t |\xi|^2} |\widehat u_0(\xi)| + \int_0^t |\xi|^2 e^{-(t-\tau)|\xi|^2} \|u(\tau)\|_{L^2}^2\,d\tau \\
&\leq \|u_0\|_{\my^1} + \|u_0\|_{L^2}^2 \int_0^t |\xi|^2 e^{-(t-\tau)|\xi|^2}\,d\tau \leq C_{0}
\end{aligned}
\]
from which we deduce that 
\[
\sup_{0\le t <\infty} \|u(t)\|_{\my^1} \leq C_{0}.
\]
When $\sigma \in [-1,1)$, we bound $u$ as (\ref{time bound 1}):
\[
\begin{aligned}
|\xi|^{\sigma} |\widehat u(t,\xi)| &\le  |\xi|^{\sigma} e^{-t |\xi|^2} |\widehat u_0(\xi)| + \int_0^t |\xi|^{1+\sigma} e^{-(t-\tau)|\xi|^2} \|u(\tau)\|_{L^2}^2\,d\tau \\
&\le \|u_0\|_{\my^{\sigma}} + C_{0}\int_0^t (t-\tau)^{-\frac{1+\sigma}{2}}(1+\tau)^{-\frac 32+\sigma}\,d\tau\leq C_{0}.
\end{aligned}
\]
So, we have 
\[
\sup_{0\le t <\infty} \|u(t)\|_{\my^\sigma} \leq C_{0}.
\]

\section*{Acknowledgments}
H. B. was supported by the National Research Foundation of Korea (NRF) grant funded by the Korea government (MSIT) (RS-2024-00341870).

J. J. was supported by the National Research Foundation of Korea (NRF) grant funded by the Korea government (MSIT) (RS-2022-00165600).

J. S. was supported by the National Research Foundation of Korea (NRF) grant funded by the Ministry of Education (RS-2025-25437873), and the Ministry of Science and ICT (RS-2024-00406821).

\end{document}